\documentclass[12pt, a4paper]{amsart}
\usepackage{amsmath}
\usepackage{geometry,amsthm,graphics,tabularx,amssymb,shapepar}
\usepackage[cmyk]{xcolor}
\usepackage{appendix}
\usepackage{amscd}
\usepackage[all]{xypic}
\usepackage{ulem}
\usepackage{bm}

\makeatletter
\newcommand*{\rom}[1]{\expandafter\@slowromancap\romannumeral #1@}
\makeatother

\newcommand{\BC}{{\mathbb {C}}}

\newcommand{\BN}{{\mathbb {N}}}

\newcommand{\BR}{{\mathbb {R}}}

\newcommand{\CF}{{\mathcal {F}}}

\newcommand{\CH}{{\mathcal {H}}}

\newcommand{\CO}{{\mathcal {O}}}

\newcommand{\RC}{{\mathrm {C}}}

\newcommand{\Hom}{{\mathrm{Hom}}}

\newcommand{\Spec}{{\mathrm{Spec}}}

\newcommand{\wt}{\widetilde}
\newcommand{\wh}{\widehat}

\newcommand{\q}{\mathfrak q}

\newcommand{\m}{\mathfrak m}

\newcommand{\C}{\mathbb{C}}
\newcommand{\R}{\mathbb R}

\newcommand{\abs}[1]{\lvert#1\rvert}

\newcommand{\cf}{\textit{cf}.~}

\newcommand{\be}{\begin {equation}}
\newcommand{\ee}{\end {equation}}
\newcommand{\bee}{\begin {equation*}}
\newcommand{\eee}{\end {equation*}}

\renewcommand{\mid}{\,:\,}

\theoremstyle{Theorem}

\theoremstyle{Theorem}

\theoremstyle{Theorem}

\theoremstyle{Theorem}

\theoremstyle{Plain}

\theoremstyle{remark}

\theoremstyle{remark}

\theoremstyle{Definition}
\newtheorem{dfn}{Definition}[section]

\newtheorem{cord}[dfn]{Corollary}
\newtheorem{prpd}[dfn]{Proposition}
\newtheorem{thmd}[dfn]{Theorem}
\newtheorem{lemd}[dfn]{Lemma}
\newtheorem{remarkd}[dfn]{Remark}
\newtheorem{exampled}[dfn]{Example}
\numberwithin{equation}{section}

\begin{document}

	\title[Formal manifolds]{Formal manifolds: foundations}
	
	\author[F. Chen]{Fulin Chen}
	\address{School of Mathematical Sciences, Xiamen University,
		Xiamen, 361005, China} \email{chenf@xmu.edu.cn}

	\author[B. Sun]{Binyong Sun}
	\address{Institute for Advanced Study in Mathematics \& New Cornerstone Science Laboratory, Zhejiang University,  Hangzhou, 310058, China}
	\email{sunbinyong@zju.edu.cn}
	
	\author[C. Wang]{Chuyun Wang}
	\address{Institute for Theoretical
		Sciences/ Institute of Natural Sciences, Westlake University/ Westlake Institute for Advanced Study,
		Hangzhou, 310030, China}
	\email{wangchuyun@westlake.edu.cn}
	
	
	\begin{abstract}
		This is the first paper in a series that studies 
		smooth relative Lie algebra homologies and cohomologies
		based on the theory of formal manifolds and formal Lie groups.
		In this paper, we lay the foundations for this study by introducing the notion of formal manifolds in the context of differential geometry, inspired by the notion of formal schemes in algebraic geometry.
		We develop the basic theory for formal manifolds, and establish a fully faithful
		contravariant functor from the category of formal manifolds to the category of
		topological $\C$-algebras. We also prove the existence of finite products in the category of formal manifolds by studying vector-valued formal functions.
	\end{abstract}
	
	\maketitle
	
	\tableofcontents
	
	\section{Introduction and the main results}
	
	In the setting of algebraic geometry,
	formal schemes are intensively studied in the literature.
	In this article, we study similar objects called formal manifolds in the setting of differential geometry, motivated by their applications to the representation theory of Lie groups.
	
	\subsection{The motivation}
	The theory of smooth cohomologies and smooth homologies for Lie group representations are respectively introduced in  \cite{HM} and \cite{BW}. %which plays an important role in the the study of representation theory for Lie groups.
	On the other hand, 
	%for the study of Harish-Chandra modules, 
	there is a vast literature for the 
	algebraic theory of relative Lie algebra (co)homologies, which is concerned with the (co)homologies of representations
	of certain Lie pairs $(\q,L)$. Here $\q$ is a finite-dimensional complex Lie
	algebra and $L$ is a compact  Lie group,  with additional structures that satisfy specific compatibility conditions (see \cite[(1.64)]{KV} for more details). %\cite[Chapter 5.3]{HP}
	Motivated by the recent study in representation theory and automorphic forms, there is a growing interest in establishing a smooth
	(co)homology theory for representations of
	general Lie pairs $(\q,L)$, with $L$ not necessarily compact. We refer to this theory as smooth relative Lie algebra (co)homology theory, which unifies the smooth (co)homology theory (as introduced in  \cite{HM} and \cite{BW}) and the algebraic theory of relative Lie algebra (co)homologies.

	We expect that the theory of  smooth relative (co)homologies will be useful in the study of
	representation theory of
	real reductive groups.
	For example, by generalizing the work of Wong (see \cite{Wo1,Wo2}), the theory should yield interesting smooth representations of real reductive groups through
	cohomological induction.
	This is the analytic analog of Zuckerman's functor that
	produces interesting Harish-Chandra modules (see \cite{KV}).
	However, in Zuckerman's construction, no topology is considered, whereas the topologies that we put on the representations should play an important role in some problems in representation theory.
	
	Starting with this paper, we will develop a new theory of smooth relative (co)homologies in several papers.
	In this paper, we introduce and study formal manifolds,
	which are analogs of formal schemes in the setting of differential geometry.
	The definition of formal manifolds naturally comes out of the two objects in Lie pairs.
	Recall that the formal Lie theory theorem asserts an equivalence of categories between
	the category of finite-dimensional complex Lie algebras and the category
	of complex formal group laws (see \cite[(14.2.3)]{Ha} for example). Underlying a formal group law is a formal power series algebra.
	On the other hand, a smooth manifold (and in particular a Lie group) is naturally a locally ringed space with the sheaf of
	$\C$-valued smooth functions as its structure sheaf.
	Roughly speaking,  formal manifolds are locally ringed spaces over $\Spec(\C)$ that unify the notions of  formal power series algebras and smooth manifolds
	(see Definition \ref{def:formalmanifold} for details).
	%This unified view of  Lie pairs plays an essential role in our study of the (co)homology theory in several papers to follow.
	
	This paper primarily focuses on the foundational framework of formal manifolds. %The forthcoming part of this series will investigate the foundational framework of formal manifolds, covering theories of constant rank morphisms and formal submanifolds.
	%In this paper, we are mainly concerned about the function spaces and Poincar\'e's lemma on formal manifolds.
	%The basic structure of formal manifolds will be further studied in the next paper of this series, containing the theories of constant rank morphisms and formal submanifolds.
	In another paper of this series, we will study formal Lie groups, which are defined as group objects in the category of formal manifolds.
	%which are analogs of formal groups in the setting of differential geometry.
	%Similar to formal groups, which are group objects in the category of formal schemes, formal Lie groups are group objects in the category of formal manifolds, by definition.
	We will prove a generalization of the formal Lie theory theorem, which asserts an equivalence of categories between the category of formal Lie groups
	and that of  Lie pairs. 
	Using the language of formal manifolds and formal Lie groups, we will finally establish the smooth relative
	(co)homology theory.
	
	We will define formal manifolds and related notions in what follows, and provide an outline of the main results of this paper.
	
	\subsection{$\BC$-locally ringed spaces}
	
	As mentioned before, we will work in the framework of locally ringed spaces.
	
	\begin{dfn}
		A $\BC$-locally ringed space is a topological space $M$, together with a sheaf  $\CO$ of commutative unital $\BC$-algebras on it such that all the stalks are local rings. Equivalently, a  $\BC$-locally ringed space is a locally ringed space over the locally ringed space $\Spec(\BC)$.
	\end{dfn}

	All $\BC$-locally ringed spaces form a category as usual. The morphisms in this category are morphisms of locally ringed spaces over $\Spec(\BC)$. More precisely, a morphism from a $\BC$-locally ringed space $(M,\CO)$ to another  $\BC$-locally ringed space $(M',\CO')$ is a pair $\varphi=({\overline\varphi}, \varphi^*)$, where ${\overline\varphi}: M\rightarrow M'$ is a continuous map and
	\be\label{phin}
	\varphi^*: {\overline\varphi}^{-1}\CO'\rightarrow \CO
	\ee
	is a $\BC$-algebra sheaf homomorphism that induces local homomorphisms on the stalks.
	For  open subsets $U'$ of $M'$ and $U$ of $M$ such that $\overline\varphi(U)\subset U'$,
	we write
	\[\varphi^*_{U',U}:\CO'(U')\rightarrow \CO(U)\] for the homomorphism of $\C$-algebras  induced by \eqref{phin}.
	If there is no confusion, we will denote   $\varphi^*_{U',U}$ by  $\varphi^*_{U'}$ or  $\varphi^*$ for simplicity.

	For every $\BC$-locally ringed space $(M, \CO)$, define a subsheaf $\m_\CO$ of $\CO$ by
	\be\label{eq:defmo}
	\m_\CO(U):=\{f\in \CO(U)\mid f_a\in \m_a  \textrm{ for all $a\in U$}\},
	\ee
	where $U$ is an open subset of $M$, $f_a$ is the germ of $f$ at $a$,  and  $\m_a$ is the maximal ideal of the stalk $\CO_a$.
	Note that $\m_\CO$ is an ideal of $\CO$. Form the quotient sheaf
	\be\label{underco}
	\underline \CO:=\CO/\m_\CO.
	\ee
	Then $(M, \underline \CO)$ is also a $\BC$-locally ringed space.

	\begin{dfn}\label{defred}
		The $\BC$-locally ringed space $(M, \underline \CO)$ defined above is called the reduction of $(M, \CO)$.
		The $\BC$-locally ringed space $(M, \CO)$ is said to be reduced if $\underline \CO=\CO$, or equivalently, $\m_\CO$ is the zero sheaf.
	\end{dfn}
	
	By abuse of notation,  we will often not distinguish a $\BC$-locally ringed space  $(M, \CO)$ with its underlying topological space $M$, and call $\CO$ the structure sheaf of $M$.
	We write $\underline M:=M$ as a topological space, to be viewed as a reduced $\BC$-locally ringed space with the structure sheaf $\underline \CO$.
	
	All reduced $\BC$-locally ringed spaces form a full subcategory of the category of
	$\BC$-locally ringed spaces, and thus they also form a category.
	By using \eqref{underco}, we have an obvious morphism
	\be\label{eq:reductionmorphism}
	\underline M\rightarrow M
	\ee
	of $\BC$-locally ringed spaces. For every morphism $\varphi:M_1\rightarrow M_2$ of $\BC$-locally ringed spaces, there is a unique morphism 
	\be\label{eq:reductionvarphi}
	\underline \varphi: \underline{M_1}\rightarrow \underline{M_2}\ee  such that the diagram
	\be \label{underline }
	\begin{CD}
		M_1 @>  \varphi >> M_2\\
		@AAA          @AAA\\
		\underline{M_1}@> \underline \varphi >>  \underline{M_2} \\
	\end{CD}
	\ee
	commutes. We call $\underline \varphi$ the reduction of $\varphi$.
	The reduction assignments  $M\mapsto \underline M$ and $\varphi\mapsto \underline \varphi$ form a functor from the category of  $\BC$-locally ringed spaces to the category of reduced  $\BC$-locally ringed spaces.
	
	For a $\BC$-locally ringed space $M$ with structure sheaf $\CO$, the formal stalk at a point $a\in M$ is defined to be the local $\BC$-algebra
	\be\label{eq:fs}
	\widehat \CO_a:=\varprojlim_{k\in \BN} \CO_a/\m_a^k\qquad (\BN:=\{0,1,2,\dots\}).
	\ee 
	%Here and as before, $\m_x$ denotes the  maximal ideal of the stalk $\CO_x$.

	\subsection{Formal manifolds}
	We consider a smooth manifold as a $\BC$-locally ringed space with the sheaf of complex-valued smooth functions as the structure sheaf.
	
	Throughout this paper, we assume that all smooth manifolds are Hausdorff and paracompact. They may or may not be equi-dimensional,
	and may or may not have countably many connected components. We also view the empty set as a smooth manifold. For a smooth manifold $N$, we use $\mathrm C^\infty(N)$ to denote the $\BC$-algebra of complex-valued smooth functions on $N$.
	
	\begin{exampled}\label{trivialf}
		Let  $k\in \BN$, and let $N$ be a smooth manifold. For every open subset $U\subset N$, write
		\[
		\CO_N^{(k)}(U):=\RC^\infty(U)[[y_1, y_2, \dots, y_k]] \quad (\textrm{the algebra of formal power series}).
		\]
		With the obvious restriction maps, $(N, \CO_N^{(k)})$ is a $\BC$-locally ringed space. For simplicity, we also write $N^{(k)}$ for this $\BC$-locally ringed space.
	\end{exampled}
	
	The $\BC$-locally ringed space $N^{(k)}$ defined as above provides an example of formal manifolds. In general, we make the following definition.

	\begin{dfn}\label{def:formalmanifold}
		A formal manifold  is a $\BC$-locally ringed space $(M, \CO)$ such that
		\begin{itemize}
			\item the topological space $M$ is paracompact and Hausdorff; and
			\item for every $a\in M$, there is an open  neighborhood $U$ of $a$ in $M$ and $n,k\in \BN$ such that $(U, \CO|_U)$ is isomorphic to $(\R^n)^{(k)}$ as a $\BC$-locally ringed space.
		\end{itemize}
	\end{dfn}
	
	Let $(M, \CO)$ be a formal manifold.
	For every $a\in M$, the  uniquely determined natural numbers $n$ and $k$ in Definition \ref{def:formalmanifold} are respectively
	called the dimension and the degree of $M$ at $a$, denoted by $\dim_a M$ and $\deg_a M$.
	We say that $(M,\CO)$ has equi-dimension $n$ if $\dim_a M=n$ for all $a\in M$,
	and has equi-degree $k$ if  $\deg_a M=k$ for all $a\in M$.
	An element in $\CO(M)$ is called a formal function on $M$.
	
	All formal manifolds form a full subcategory of the category of $\BC$-locally ringed spaces, and thus they also form a category.
	
	\begin{exampled}\label{trivialf2}
		A smooth manifold is naturally a formal manifold of equi-degree $0$.
		On the other hand, the reduction of a formal manifold is naturally a smooth manifold.
		Moreover, for every  formal manifold, the following conditions are  equivalent to each other:
		\begin{itemize}
			\item
			it has equi-degree $0$;
			\item
			it is reduced in the sense of Definition \ref{defred}; 
			\item
			it is a smooth manifold.
		\end{itemize}
	\end{exampled}

	\begin{dfn}\label{df:infini}	A formal manifold is said to be infinitesimal if its underlying topological space has precisely one element.
	\end{dfn}
	
	\begin{exampled}\label{trivialf3}
		A formal power series algebra $A:=\C[[y_1, y_2, \dots, y_k]] $ ($k\in \BN$) induces naturally  an infinitesimal formal manifold $(M,\CO)$ such that  $\CO(M)= A$ as  $\C$-algebras.
		Conversely, up to isomorphism, every infinitesimal formal manifold has such a form. % Conversely, every infinitesimal formal manifold is isomorphic to such an infinitesimal formal manifold.
	\end{exampled}
	%It is clear that the isomorphism class of an infinitesimal formal manifold is determined by its degree.
	
	% In Section 2, we establish some basic 
	% properties of formal
	% manifolds, including the existence of partitions of unity on formal manifolds,
	% the structure of formal stalks, and the softness of structure sheaves.
	%  In Section 3,  we introduce notions of tangent spaces and differential maps between them.
	% Then we prove the inverse function theorem in the setting of formal manifolds.

	\subsection{Structure of the paper} Here we give an outline of the main results in this paper. 
	Let $(M,\CO)$ be a formal manifold. 
	In Section \ref{sec:pre}, we establish some basic properties of $(M,\CO)$, including the existence of partitions of unity on $M$,
	the structure of formal stalks $\wh\CO_a$ ($a\in M$), and the softness of the sheaf $\CO$.
	
	In Section \ref{sec:Diffop}, we introduce the notion of differential operators between two sheaves of $\CO$-modules. 
	When $M=N^{(k)}$ for some open submanifold $N$ of $\BR^n$ and $n,k\in \BN$, for every sheaf $\CF$ of $\CO$-modules that satisfies  two natural conditions, we provide in Proposition \ref{dofm0} an explicit construction of all
	differential operators from $\CO$ to $\CF$.
	
	In Section \ref{sec:smoothtop}, by using certain compactly supported differential operators, we define a canonical topology on the space $\CF(M)$ for every sheaf $\CF$ of $\CO$-modules.  We refer to this topology as the smooth topology.
	With the smooth topology, $\CO(M)$ becomes a topological $\BC$-algebra. In other words, the addition map, the multiplication map, and the scalar multiplication map, are all jointly continuous. 
	As a topological vector space, it is a product of nuclear Fr\'echet spaces.
	Specifically, the formal stalk $\widehat \CO_a$ is a topological $\BC$-algebra for every $a\in M$ (see Example \ref{trivialf3} and Proposition \ref{lem:formalstalkiso}).

	We prove the following theorem in Section \ref{sec:formalspec}.
	
	\begin{thmd}\label{thmmain1}
		For every morphism \[\varphi=(\overline\varphi, \varphi^*): (M, \CO)\rightarrow (M', \CO')\] of formal manifolds, the homomorphism \[\varphi^*: \CO'(M')\rightarrow \CO(M)\] between topological $\C$-algebras is continuous. Moreover,
		\be \label{eq:introfunctor}
		(M, \CO)\mapsto \CO(M), \quad (\varphi: (M, \CO)\rightarrow (M', \CO'))\mapsto (\varphi^*: \CO'(M')\rightarrow \CO(M))
		\ee
		is a fully faithful contravariant functor from the category of formal manifolds to the category of topological $\BC$-algebras.
		
	\end{thmd}

	% For two LCS  $E_1$ and $E_2$, we equip the tensor product space $E_1\otimes E_2$ with the projective tensor product, and denote the resulting LCS by $E_1\otimes_\pi E_2$. 
	% Write $E_1\widetilde\otimes E_2$ and $E_1\widehat\otimes_\pi E_2$ for the quasi-completion  and completion of $E_1\otimes_\pi E_2$,  respectively. 
	Throughout this paper, by an LCS, we mean a locally convex topological vector space over $\BC$ which may or may not be Hausdorff. For two LCS $E$ and $F$, we denote by $E\widetilde\otimes_\pi F$ and $E\widehat\otimes_\pi F$ the quasi-completed and completed projective tensor product of $E$ and $F$, respectively.
	
	We know that finite products exist in the category of smooth manifolds. In Section \ref{sec:prod}, we extend this result to the category of formal manifolds.

	\begin{thmd}\label{thmmain2}
		Finite products exist in the category of formal manifolds. Moreover, for all formal manifolds $M_1$, $M_2$, $M_3$ such that
		\[
		M_3=M_1\times M_2,
		\]
		there are  identifications
		\[
		\CO_{3}(M_3)=\CO_{1}(M_1)\widetilde \otimes_\pi \CO_{2}(M_2)=\CO_{1}(M_1)\widehat \otimes_\pi \CO_{2}(M_2)\]
		and
		\[ \widehat \CO_{3,(a_1,a_2)}=\widehat \CO_{1,a_1}\widetilde \otimes_\pi \widehat \CO_{2,a_2}=\widehat \CO_{1,a_1}\widehat \otimes_\pi \widehat \CO_{2,a_2}\quad ((a_1,a_2)\in M_3=M_1\times M_2)
		\]
		of topological $\C$-algebras, and an identification
		\bee
		\underline{M_3}=\underline{M_{1}} \times\underline{M_{2}}
		\eee
		of smooth manifolds. Here $\CO_i$  ($i=1,2,3$) is the structure sheaf of $M_i$.  
	\end{thmd}

	%In the third paper of this series, the above Poincar\'e's lemmas for formal manifolds will be used to construct the standard (projective and injective) resolutions for representations of general Lie pairs, which is a stone step for the theory of smooth relative Lie algebra (co)homologies.

	\section{Preliminaries on formal manifolds}\label{sec:pre}
	
	In this section, we introduce some basic notions and establish some basic properties for formal manifolds.

	\subsection{Notations and conventions} Here we list some notations and conventions which will be used later.
	
	In the rest of this paper, we let $(M,\CO)$ denote a formal manifold. 
	We say that the formal manifold $(M,\CO)$ is secondly countable if so is its underlying topological space. 
	We denote by $\pi_0(M)$ the set of all connected components in $M$.
	
	We say that an open subset $U$ of $M$ is a chart if  there exist $n,k\in \BN$ and an open
	submanifold $N$ of $\R^n$ such that $(U,\CO|_U)\cong (N,\CO_N^{(k)})$ as $\C$-locally ringed spaces. 
	An open cover of $M$ consisting of charts will be called an atlas of $M$.
	
	Let $a\in M$.
	Recall that $\dim_a M$ and $\deg_a M$ denote the dimension and degree of $M$ at $a$,  respectively.
	For every sheaf  $\CF$ of $\CO$-modules, we let $\CF_a$ denote the stalk of $\CF$ at $a$.
	We let $\m_a$ denote the maximal ideal of the stalk $\CO_a$, and define the Dirac
	distribution supported at $a$ to be the
	linear functional
	\be
	\label{eq:defdeltax}
	\delta_a:\ \mathcal{O}_a\longrightarrow \mathcal{O}_a/\m_a=\C.\ee
	
	As in the Introduction, for every open subset $U$ of $M$,
	we call an element of $\CO(U)$ a formal function on $U$.
	For a formal function $f$ on $U$ and a point $a$ in $U$,
	we define the value  $f(a)\in \BC$ of $f$ at $a$ to be
	the image of $f$ under the following composition of natural maps:
	\[\mathcal{O}(U)\longrightarrow \mathcal{O}_a\stackrel{\delta_a}{\longrightarrow} \C.\]
	We also define the reduction
	\be\label{underf}
	\underline{f}\in \underline{\CO}(U)
	\ee
	of $f$ to be the image of $f$ under the reduction map 
	$\CO(U)\longrightarrow \underline{\CO}(U)$ (see \eqref{eq:reductionmorphism}).
	It is clear that $f(a)=\underline{f}(a)$ for all $a\in U$.
	
	\begin{exampled}\label{ex:formalfunction} If $M=N^{(k)}$ for some smooth manifold $N$ and $k\in \BN$, then the reduction of
		\[
		f=\sum_{(j_1,j_2,\dots,j_k)\in \BN^k}f_{(j_1,j_2,\dots,j_k)}\, y_1^{j_1}y_2^{j_2}\cdots y_k^{j_k}\in \RC^\infty(U)[[y_1,y_2,\dots,y_k]]
		\]
		is $f_{(0,0,\dots,0)}$, and the value of $f$ at $a$ is $f_{(0,0,\dots,0)}(a)$.
	\end{exampled}

	Suppose now that $U$ is an open subset of $\R^n$.
	We denote the standard coordinate functions on $U$ as $x_1, x_2, \dots, x_n$, and the first-order partial derivatives with respect to these variables as
	\[\partial_{x_1},\partial_{x_2},\dots,\partial_{x_n}.\]
	For every $n$-tuple $I=(i_1,i_2,\dots,i_n)\in \BN^n$, we use  the following usual multi-index notations:
	\be\label{eq:multi} \begin{array}{rl}
		I!=i_1!i_2! \cdots  i_n!, &  |I|=i_1+i_2+\cdots+i_n, \\
		x^I=x_1^{i_1}x_2^{i_2}\cdots x_n^{i_n},  &  \partial_x^I=\partial_{x_1}^{i_1}\partial_{x_2}^{i_2}\cdots \partial_{x_n}^{i_n}.
	\end{array}
	\ee
	We also view the $|I|$-th order partial derivative $\partial_x^I$ as an operator on the formal power series algebra $\RC^\infty(U)[[y_1,y_2,\dots,y_k]]$ in
	an obvious way.
	Let
	\[\partial_{y_1},\partial_{y_2},\dots,\partial_{y_k}\] denote the first-order  partial  derivatives on  $\RC^\infty(U)[[y_1,y_2,\dots,y_k]]$
	with respect to the formal variables $y_1,y_2,\dots,y_k$.
	For every $k$-tuple $J\in \BN^k$, the factorial $J!$, the length $|J|$, the monomial $y^J$, and the $|J|$-th order partial  derivative $\partial_y^J$, are defined as in \eqref{eq:multi}.

	\subsection{Partition of unity}
	In this subsection, we prove that partitions of unity exist for formal manifolds.

	\begin{lemd}\label{lem:inverse} Let $f$ be a formal function on $M$. Then the following conditions are
		equivalent to each other:
		\begin{itemize}
			\item $f$ is invertible;
			\item the reduction $\underline{f}$ of $f$ is invertible;
			\item the value of $f$ at $a$ is nonzero for every $a\in M$.
		\end{itemize}
	\end{lemd}
	\begin{proof} The assertion follows from  Example \ref{ex:formalfunction} and the sheaf property of $\CO$ and $\underline {\CO}$.
	\end{proof}
	% 2022.12.21æ³¨è®°ï¼æ­¤å¥ç¼©æï¼å¯ä¹ï¼?
	Let $\mathcal{U}=\{U_\gamma\}_{\gamma\in \Gamma}$ be an open cover
	of $M$. A family $\{f_\gamma\in \CO(M)\}_{\gamma\in \Gamma}$  of formal functions is said to be a
	partition of unity on $M$  subordinate to $\mathcal{U}$ if the following  conditions hold:
	\begin{itemize}
		\item $\mathrm{supp}\,f_\gamma$ is contained in $U_\gamma$ for all $\gamma\in \Gamma$;
		\item the family $\{\mathrm{supp}\,f_\gamma\}_{\gamma\in \Gamma}$ is locally finite; and
		\item $\sum_{\gamma\in \Gamma } f_{\gamma}=1$, that is, for every relatively compact open subset $U$ of $M$,
		\[
		\sum_{\gamma\in \Gamma; (\mathrm{supp}\,f_\gamma)\cap U\neq \emptyset } (f_{\gamma})|_U=1.
		\]
	\end{itemize}
	Here and henceforth, ``$\mathrm{supp}$" indicates the support of a section of a sheaf.

	When no confusion is possible, 
	we write $\overline U$ for the closure of a subset $U$ in a topological space. 
	Similar to the smooth manifold case, we have the following result.
	\begin{prpd}\label{lemr01}
		For every open cover
		of $M$, there is a  partition of unity on $M$ subordinate to it.
	\end{prpd}
	
	\begin{proof}
		The proof is similar to that of the smooth manifold case (see \cite[Theorem 10.1]{Br} for example). We sketch a proof for %delete the  
		completeness.
		Let $\mathcal{U}=\{U_\gamma\}_{\gamma\in \Gamma}$ be an open cover
		of $M$. Replacing $\mathcal{U}$ by a   refinement of it if necessary, we assume without loss of generality  that $\mathcal{U}$ is locally finite and
		$U_\gamma$ is relatively compact for all $\gamma\in \Gamma$.
		
		Take an open cover $\{V_\gamma\}_{\gamma\in \Gamma}$ of $M$ such that the closure  $ \overline {V_{\gamma}}$  of $V_\gamma$ is contained in $U_\gamma$ for all $\gamma\in \Gamma$.
		Then  for each $\gamma\in \Gamma$, there is a formal function $f'_\gamma\in \CO(M)$ such that
		\begin{itemize}
			\item the support of $f'_\gamma$ is contained in $U_\gamma$; and
			\item $f'_\gamma(a)$ is a non-negative real number for all $a\in M$, and is a positive real number for all $a$ in the compact set $\overline {V_\gamma}$.
		\end{itemize}
		It is obvious that  \[f_\Gamma:=\sum_{\gamma\in \Gamma} f'_{\gamma}\] is a well-defined formal function on $M$ such that
		$f_\Gamma(a)>0$ for all $a\in M$.
		Lemma   \ref{lem:inverse} implies that the formal function $f_{\Gamma}$ is invertible.
		The lemma then follows as  $\{f_\gamma\}_{\gamma\in \Gamma}$ is a partition of unity on $M$ subordinate to $\mathcal{U}$, where $f_\gamma:=f_\gamma'\cdot f_\Gamma^{-1}$.
	\end{proof}
	
	For every sheaf $\CF$ of abelian groups on a topological space $X$, write ${\CF}|_S$ for its pull-back through the inclusion map $S\rightarrow X$, where $S$ is a subset of $X$. For every section $f\in \CF(X)$, write \[f|_S\in  {\CF}|_S(S)\] for the image of $f$ under the natural map $\CF(X)\rightarrow {\CF}|_S(S)$.
	We will use the following result frequently without further explanation.
	
	\begin{cord}\label{cor:bumpfun}
		Let $V$, $U$ be two open subsets of $M$ such that $\overline{V}\subset U$. Then there are   formal functions $f,g\in \CO(M)$ such that
		\begin{eqnarray*}f|_{\overline{V}}&=1,\quad\text{and}\quad f|_{M\setminus U}&=0 ,\\
			g|_{\overline{V}}&=0,\quad\text{and}\quad g|_{M\setminus U}&=1.\end{eqnarray*}
	\end{cord}
	\begin{proof}
		The assertion follows by taking $\{f,g\}$ to be a partition of unity on $M$ subordinate to the open cover $\{U,M\setminus \overline{V}\}$.
	\end{proof}
	
	\subsection{Formal stalks}\label{subsec:formalstalk}
	Let $a\in M$.
	Recall from the Introduction that the formal stalk of $\CO$ at $a$ is
	\[\wh{\CO}_a:=\varprojlim_{r\in \BN} \CO_a/\m_a^r.\]
	For a formal function $f$ defined on an open neighborhood of $a$,
	write $f_a$ for the germ of $f$ in $\CO_a$. 
	\begin{lemd}\label{lem:mrbasis} Assume that $M=(\R^n)^{(k)}$ for some $n,k\in \BN$. Then
		for every $a\in M$ and $r\in \BN$, the set
		\be\label{eq:basis}
		\{((x-a)^Iy^J)_a+\m_a^r\mid I\in \BN^n, J\in \BN^k\ \text{with}\ |I|+|J|<r\}
		\ee
		forms a basis of $\CO_a/\m_a^r$, where for $I=(i_1,i_2,\dots,i_n)$ and $J=(j_1,j_2,\dots,j_k)$,
		\[
		(x-a)^Iy^J:=(x_1-x_1(a))^{i_1}(x_2-x_2(a))^{i_2}\cdots(x_n-x_n(a))^{i_n}y_1^{j_1}y_2^{j_2}\cdots y_k^{j_k}.
		\]
	\end{lemd}
	\begin{proof}
		Let $U$ be an open neighborhood of $a$ in $M$, and let
		\[f=\sum_{J\in \BN^k} f_J y^J\in \RC^\infty(U)[[y_1,y_2,\dots,y_k]].\]
		For every $J\in \BN^k$ with $|J|<r$, write
		\[f_J=\sum_{I\in \BN^n; |I|< r-|J|}b_{I,J}\cdot (x-a)^I+ \sum_{I\in \BN^n; |I|=r-|J|}g_{I,J}\cdot (x-a)^I\]
		for the $(r-|J|-1)$-th order Taylor expansion of $f_J$ around the point $a$, where  $b_{I,J}\in \C$ and $g_{I,J}\in \RC^\infty(U)$.
		It is straightforward to check that
		\[f_a=\sum_{I\in \BN^n, J\in \BN^k; |I|+|J|< r} b_{I,J} \cdot ((x-a)^Iy^J)_a \mod\ \m_a^r.\]
		This implies that the set \eqref{eq:basis} spans $\CO_a/\m_a^r$.
		On the other hand,    by applying the
		operators
		\[\partial_x^I\partial_y^J\quad (I\in \BN^n, J\in \BN^k\ \text{with}\ |I|+|J|<r)\] to the elements in \eqref{eq:basis}, we see that these elements are  linearly   independent.
		This completes the proof.
	\end{proof}

	For every $a\in M$ and $r\in \BN$, Lemma \ref{lem:mrbasis} implies that the quotient space $\CO_a/\m_a^r$ is  finite-dimensional. We equip the space $\CO_a/\m_a^r$ with the Euclidean topology, which is the unique topology that makes it a Hausdorff LCS. 
	Now, we endow the formal stalk \[\wh{\CO}_a=\varprojlim_{r\in \BN} \CO_a/\m_a^r\]
	with the projective limit topology. Then it becomes a topological
	$\C$-algebra. 
	On the other hand, we view $\C[[x_1,x_2,\dots,x_n,y_1,y_2,\dots,y_k]]$ ($n,k\in \BN$) as a topological
	$\C$-algebra with the usual term-wise convergence topology.

	\begin{prpd}\label{lem:formalstalkiso}
		For every $a\in M$, we have that
		\be \label{eq:formalstalkiso}
		\wh{\CO}_a \cong \C[[x_1,x_2,\dots,x_n,y_1,y_2,\dots,y_k]]
		\ee
		as topological $\C$-algebras, where $n:=\dim_a M$ and $k:=\deg_a M$. Furthermore, the natural map
		\be\label{eq:stalktoformalstalk}
		\CO_a/\cap_{r\in \BN}\m_a^r\rightarrow \wh{\CO}_a
		\ee
		is a $\BC$-algebra isomorphism.
	\end{prpd}
	
	\begin{proof}
		Without loss of generality, we assume that $M=(\R^n)^{(k)}$ and  $a=0\in M$. Taylor expansion around the point $a$ yields  the following $\C$-algebra homomorphism
		\begin{equation}\begin{split}\label{eq:Taylorexpansion}
				\CO_a&\rightarrow \C[[{\mathbf{x},\mathbf{y}}]]:=\C[[x_1,x_2,\dots,x_n,y_1,y_2,\dots,y_k]],\\
				f_a&\mapsto \sum_{J\in \BN^k}\left(\sum_{I\in \BN^n}\frac{1}{I!} \partial_x^{I}(f_{J})(a) \cdot x^I\right)y^J
		\end{split}\end{equation}
		where $f=\sum_{J\in \BN^k} f_J y^J$ is a formal function on an open neighborhood of $a$
		and \[\sum_{I\in \BN^n}\frac{1}{I!} \partial_x^{I}(f_{J})(a)\cdot x^I\] is the Taylor series of $f_J$ at $a$.
		Note that Borel's lemma implies that the map \eqref{eq:Taylorexpansion} is surjective.
		
		Write $\m$ for the maximal ideal of $ \C[[{\mathbf{x},\mathbf{y}}]]$. By  Lemma \ref{lem:mrbasis}, we know that  the  algebra homomorphism
		\[
		\CO_a/\m_a^r \longrightarrow \C[[{\mathbf{x},\mathbf{y}}]]/\m^r\quad (r\in \BN)
		\]
		induced by \eqref{eq:Taylorexpansion} is an isomorphism.
		This produces a topological algebra isomorphism
		\be\label{eq:formalstalkisoring}
		\wh{\CO}_a=\varprojlim_{r\in \BN}\CO_a/\m_a^r \longrightarrow \C[[{\mathbf{x},\mathbf{y}}]]
		= \varprojlim_{r\in \BN}\C[[{\mathbf{x},\mathbf{y}}]]/\m^r,
		\ee
		which proves the first assertion in this proposition.
		For the second one, one only needs to note that
		the surjective map \eqref{eq:Taylorexpansion} is the composition of the maps
		\[\CO_a\longrightarrow \wh\CO_a\xrightarrow{\eqref{eq:formalstalkisoring}} \C[[{\mathbf{x},\mathbf{y}}]].\]
	\end{proof}

	We also record the following straightforward result as a lemma for later use.
	
	\begin{lemd}\label{dformalf0}
		The canonical map
		\be\label{maptofs}
		\CO(M)\rightarrow \prod_{a\in M} \widehat \CO_a
		\ee
		is injective.
	\end{lemd}

	Let $\varphi=(\overline{\varphi}, \varphi^*): (M,\CO)\rightarrow (M', \CO')$ be a morphism of formal manifolds.
	For every  $a\in M$,  we let
	\[\varphi^*_a:\ \CO'_{\overline{\varphi}(a)}\rightarrow \CO_a\] denote the induced local homomorphism of the stalks.
	Then there is  a unique local homomorphism
	\be\label{eq:homonformalstalks}
	\varphi^*_a:\ \widehat {\CO'}_{\overline{\varphi}(a)}\rightarrow \widehat \CO_a
	\ee
	between the formal stalks such that the diagram
	\be \label{eq:comofwhO} \begin{CD}
		\CO'_{\overline\varphi(a)} @>{\varphi^*_a}>>
		\CO_a \\@VVV @VVV\\
		\wh{\CO'}_{\overline\varphi(a)} @>{\varphi^*_a}>>
		\wh\CO_a           \\
	\end{CD}\ee
	commutes. Furthermore,
	we have the following result.
	
	\begin{lemd}\label{lem:homonformalstalkscon}
		Under the projective limit topologies, the map \eqref{eq:homonformalstalks} is continuous.
	\end{lemd}
	\begin{proof} This follows from the fact that every linear map between finite-dimensional Hausdorff LCS is automatically continuous.
	\end{proof}

	%The following result asserts that a morphism of formal manifolds is determined by its induced local homomorphisms on formal stalks.
	\begin{lemd}\label{homfst}
		Let
		\[
		\varphi=(\overline\varphi, \varphi^*),\, \psi=(\overline\psi, \psi^*):\ (M,\CO)\rightarrow (M', \CO')
		\]
		be two morphisms of formal manifolds. If $\overline\varphi=\overline\psi$, and for all $a\in M$,
		the maps
		\[
		\varphi^*_a,\, \psi^*_a:\ \widehat{\CO'}_b\rightarrow \widehat \CO_a\qquad (b:=\overline\varphi(a)=\overline\psi(a))
		\]
		are equal, then $\varphi=\psi$.
	\end{lemd}
	\begin{proof}
		This follows from Lemma \ref{dformalf0}.
		
	\end{proof}

	For every sheaf $\CF$ of $\CO$-modules  and $a\in M$, we define the formal stalk of $\CF$ at $a$ by 
	\[\widehat{\CF}_a:=\varprojlim_{r\in \BN} \CF_a/\m_a^r\CF_a.\]

	\begin{exampled}\label{ex:ol} 
		Suppose that $N$ is a smooth manifold and $k,l\in \BN$. Then there is a homomorphism $\CO_N^{(k)}\rightarrow \CO_N^{(l)}$ of sheaves of algebras given by
		\[
		\begin{array}{rcl}
			\CO_{N}^{(k)}(U)&\rightarrow &\CO_{N}^{(l)}(U), \smallskip \\
			\sum\limits_{J\in \BN^k} f_J y^J&\mapsto&\left\{ \begin{array}{ll}
				\sum\limits_{J\in \BN^k=\BN^k\times\{0\}^{l-k}\subset \BN^l}f_J y^{J},&\textrm{if $k\leq l$;} \smallskip \\
				\sum\limits_{J\in \BN^l=\BN^l\times\{0\}^{k-l}\subset \BN^k}f_Jy^{J},&\textrm{if $k>l$},
			\end{array}
			\right. 
		\end{array}
		\]
		where $U$ is an open subset of $N$ and $f_J\in \RC^\infty(U)$. By using this homomorphism, every sheaf of
		$\CO_N^{(l)}$-modules is also a sheaf of  $\CO_N^{(k)}$-modules. 
		Suppose that $\CF= \CO_{N}^{(l)}$, to be viewed as a sheaf of  $\CO_{N}^{(k)}$-modules as above. Then, 
		similar to \eqref{eq:formalstalkiso}, for every $a\in N$, we have that
		\be\label{eq:formalstalkOnl}
		\widehat{\CF}_a\cong\BC[[x_1, x_2, \dots, x_n, y_1, y_2, \dots, y_l]],\quad \textrm{where } \ n:=\dim_a N.\ee
	\end{exampled}

	%Assume that $M=(\R^n)^{(k)}$ for some $n,k\in \BN$. For every $l\in \BN$, we endow an $\CO_{\BR^n}^{(k)}$-module structure on $\CO_{\BR^n}^{(l)}$   in the following way:

	\subsection{Softness}
	We discuss here the softness of structure sheaves of formal manifolds. Recall that a sheaf $\CF$  of abelian groups on a topological space $X$ is said to be soft if
	for every closed subset $Z$ of $X$, the canonical map $\CF(X)\rightarrow (\CF|_Z)(Z)$ is surjective.
	
	\begin{prpd}\label{softo}
		The sheaf  $\CO$  is soft.
	\end{prpd}
	\begin{proof}
		By \cite[Chapter II, Lemma 9.14]{Br2}, it suffices to show that every $a\in M$ admits
		an open neighborhood $U$ such that the sheaf $\CO|_U$ is soft.
		Thus it is enough to show that the sheaf $\CO_{\R^n}^{(k)}$ ($n,k\in \BN$) is soft.
		We know that the  sheaf $\CO_{\R^n}^{(0)}$ of  smooth functions on $\R^n$ is soft (see \cite[Chapter II, Example 9.4]{Br2}).
		Recall from Example \ref{ex:ol} that $\CO_{\R^n}^{(k)}$ is a sheaf of  $\CO_{\R^n}^{(0)}$-modules.
		Thus \cite[Chapter II, Theorem 9.16]{Br2} implies that the sheaf  $\CO_{\R^n}^{(k)}$ is soft,
		as required.
	\end{proof}
	
	\begin{cord}\label{softo1}
		For every $a\in M$, the canonical maps 
		\[\CO(M)\rightarrow \CO_a\quad \text{and} \quad \CO(M)\rightarrow \wh{\CO}_a\] are both surjective.
	\end{cord}
	\begin{proof}
		This  follows from Propositions \ref{softo} and \ref{lem:formalstalkiso}.
	\end{proof}
	\begin{cord}\label{softo2}
		All sheaves of $\CO$-modules are soft.
	\end{cord}
	\begin{proof}
		This follows from Proposition \ref{softo} and  \cite[Chapter II, Theorem 9.16]{Br2}.
	\end{proof}

	Recall the reduction $\underline \CO$ of $\CO$ from \eqref{underco}.
	
	\begin{prpd}\label{prop:reductionsur}
		The canonical map
		\be\label{suroo}
		\CO(M)\rightarrow \underline \CO(M)
		\ee
		is surjective.
	\end{prpd}
	\begin{proof}
		Consider the exact sequence
		\[
		0\rightarrow \m_\CO\rightarrow \CO\rightarrow \underline \CO\rightarrow 0.
		\]
		By Corollary \ref{softo2}, the sheaf $\m_\CO$ is soft. Then the proposition follows from \cite[Chapter II, Theorem 9.9]{Br2}.
		
	\end{proof}

	\section{Differential operators on formal manifolds}\label{sec:Diffop}
	In this section, we study differential operators between two sheaves of $\CO$-modules.
	
	\subsection{Differential operators of modules}
	We first review the notion of differential operators of modules of a commutative unital $\BC$-algebra.

	Let $A$ be a commutative unital $\BC$-algebra, and let $F_1$, $F_2$ be two $A$-modules.
	For  $D\in \Hom_{\BC}(F_1,F_2)$ and  $a\in A$, define an element $[D,a]\in \Hom_{\BC}(F_1,F_2)$ by
	\[
	[D,a](u)=D(au)-aD(u)\quad \textrm{ for all }  u\in F_1.
	\]
	Let $r\in \BN$  throughout this section. 
	\begin{dfn}
		A map $D\in  \Hom_{\BC}(F_1,F_2)$ is called a  differential
		operator of order $\leq r$ if
		\[
		[\cdots[[D,a_0],a_1],\cdots, a_r] =0\ \  \textrm{  for all $a_0,a_1, \dots,a_r\in A$}.
		\]
	\end{dfn}

	Write $\mathrm{Diff}_r(F_1,F_2)$ for the space of all differential operators of order $\leq r$ in  $\Hom_{\BC}(F_1,F_2)$.
	Note that $\mathrm{Diff}_0(F_1,F_2)=\Hom_A(F_1,F_2)$, and
	\[
	\mathrm{Diff}_{r_1}(F_2,F_3)\circ \mathrm{Diff}_{r_2}(F_1,F_2)\subset \mathrm{Diff}_{r_1+r_2}(F_1,F_3)\quad \textrm{for all }r_1, r_2\in \BN,
	\]
	where $F_3$ is another $A$-module.
	\begin{dfn}
		An element $D\in \Hom_{\BC}(A,F_2)$ is called a derivation  if
		\[
		D(a_1 a_2)=a_2 D(a_1)+a_1  D(a_2)\ \  \textrm{  for all $a_1,a_2\in A$}.\]
	\end{dfn}

	Write $\mathrm{Der}(A,F_2)$ for the space of all derivations in $\Hom_{\BC}(A,F_2)$. Obviously, $\mathrm{Der}(A,F_2)\subset \mathrm{Diff}_1(A,F_2)$.
	We identify   $\mathrm{Diff}_0(A,F_2)=\Hom_{A}(A,F_2)$ with $F_2$  via the isomorphism
	\[
	\Hom_{A}(A,F_2)\rightarrow F_2, \quad D\mapsto D(1).
	\]
	One checks that $D-D(1)$ is a derivation for every $D\in \mathrm{Diff}_1(A, F_2)$, which implies 
	\be\label{diff1}
	\mathrm{Diff}_1(A,F_2)=\mathrm{Der}(A,F_2)\oplus F_2.
	\ee
	
	\begin{lemd}\label{dlocal0}
		For every  $D\in \mathrm{Diff}_r(F_1,F_2)$ and every 
		ideal $I$ of $A$, we have that
		\be\label{b}
		D(I^{i+r}F_1)\subset I^iF_2\qquad \textrm{for all }i\in \BN.
		\ee
	\end{lemd}
	
	\begin{proof}We will prove the lemma by induction on $i+r$. The lemma is obvious if $i+r=0$.
		Now assume that $i+r\geq 1$ and the lemma holds when $i+r$ is smaller.
		If $i=0$ or $r=0$, then \eqref{b} clearly holds. So we assume that $i>0$ and $r>0$. Let $f_1, f_2, \dots, f_{i+r}\in I$.
		From the fact $[D, f_1]\in \mathrm{Diff}_{r-1}(F_1,F_2)$, it follows that
		\begin{eqnarray*}
			D(f_1 f_2 \cdots f_{i+r}F_1)&\subset& D (f_1I^{i+r-1}F_1)\\
			&\subset& [D,f_1](I^{i+r-1}F_1)+ f_1D(I^{i+r-1}F_1)\\
			&\subset& I^iF_2+f_1 I^{i-1}F_2=I^iF_2.
		\end{eqnarray*}
		This proves the lemma.
	\end{proof}
	
	Assume now that $A=\BC[[y_1, y_2, \dots, y_k]]$ for some $k\in \BN$.
	For every $J\in \BN^k$ with $|J|\le r$,  we have that
	\[F_2\circ \partial_y^J\subset  \mathrm{Diff}_r(A,F_2),\]
	where the identification $F_2=\mathrm{Diff}_0(A,F_2)$ is used. Conversely, we have the following result (\cf \cite[Theorem 1.1.8]{Bj}).
	
	\begin{lemd}\label{diff}
		Assume that $A=\BC[[y_1, y_2, \dots, y_k]]$ for some $k\in \BN$, and $F_2$ is  an $A$-module such that
		$\cap_{i\in \BN}\m^i F_2=0$, where $\m$ denotes the maximal ideal in $A$.
		Then
		\[
		\mathrm{Diff}_r(A,F_2)=\bigoplus_{J\in \BN^k; \abs{J}\leq r} F_2\circ \partial_y^J.
		\]
	\end{lemd}

	\begin{proof} Denote by $\C[\mathbf{y}]_{\le r}\subset \C[y_1,y_2,\dots,y_k]$ the subspace of polynomials of degree $\le r$ in $A$. We claim that
		\begin{equation}\label{psirinj}
			\text{the restriction map}\ \Psi_r:\ \mathrm{Diff}_r(A,F_2)\rightarrow \Hom_\BC(\C[\mathbf{y}]_{\le r}, F_2)\ \text{is injective}.\end{equation}
		Note that the lemma is implied by this claim, as
		\[\Psi_r\left(\bigoplus_{J\in \BN^k; \abs{J}\leq r} F_2\circ \partial_y^J\right)=\mathrm{Hom}_\C(\C[\mathbf{y}]_{\le r},F_2).\]

		We now prove the claim \eqref{psirinj} by induction on $r$.
		The claim obviously holds when $r=0$, and so we assume that $r>0$.
		Pick an element $D\in \mathrm{Diff}_r(A,F_2)$ such that $D(f)=0$ for all $f\in \C[\mathbf{y}]_{\le r}$.
		Then for every $i=1,2,\dots,k$, we have  that
		\[[D,y_i]\in \mathrm{Diff}_{r-1}(A,F_2)\quad\text{and}\quad [D,y_i](g)=0\ \text{for all}\ g\in \C[\mathbf{y}]_{\le r-1}.\]
		These together with the induction hypothesis imply that $[D,y_i]=0$ on $A$ for all $i$, and then one checks that 
		\begin{equation}\label{Dvanishpoly}
			D(h)=0\quad \text{for all}\ h\in \BC[y_1, y_2,\dots, y_k].
		\end{equation}
		Let $f=\sum_{J\in \BN^k}c_J y^J$ be an arbitrary formal power series in $A$.
		By  \eqref{Dvanishpoly} and Lemma \ref{dlocal0},  for every $j\in \BN$, we have that
		\[D(f)=D(f)-D\left(\sum_{|J|< j+r}c_J y^J\right)=D\left(\sum_{|J|\ge j+r}c_J y^J\right)\in D(\m^{j+r})\subset \m^j F_2.\]
		Thus, the condition  $\cap_{j\in \BN}\m^j F_2=0$ in the lemma forces   that $D(f)=0$.
		This proves  \eqref{psirinj} and hence the lemma.
	\end{proof}
	
	\subsection{Differential operators on formal manifolds}\label{sec:diff}
	Let $\CF_1$ and $\CF_2$ be two sheaves of $\mathcal{O}$-modules. In this subsection, we introduce the
	notion of differential operators between $\CF_1$ and $\CF_2$.
	
	Recall that $\C_M$ denotes the sheaf  of  locally constant $\BC$-valued functions on $M$.
	Write
	$\Hom_{\BC_M}(\CF_1,\CF_2)$ for the space of all $\BC_M$-homomorphisms between $\CF_1$ and $\CF_2$.
	Then the assignment
	\[U\mapsto\Hom_{\BC_U}(\CF_1|_U,\CF_2|_U)\quad (\text{$U$ is an open subset of $M$})\]
	forms a sheaf of $\BC$-vector spaces over $M$, to be denoted by  $\CH\textrm{om}_{\BC_M}(\CF_1,\CF_2)$.

	Recall that an element $D$ of $\Hom_{\BC_M}(\CF_1,\CF_2)$ is defined to be a family 
	\[
	\{
	(D_U: \CF_1(U)\rightarrow \CF_2(U))\}_{U\textrm{ is an open subset of }M}
	\]
	of linear maps that commute with the restriction maps.  
	\begin{dfn}
		An element $D\in\Hom_{\BC_M}(\CF_1,\CF_2) $ is called  a differential operator of order $\leq r$ if
		\[
		D_U\in \mathrm{Diff}_r(\CF_1(U),\CF_2(U))
		\]
		for all open subsets $U$ of $M$, where $\CF_1(U)$ and $\CF_2(U)$ are viewed as $\CO(U)$-modules, and $\CO(U)$ is viewed as a $\C$-algebra.    
	\end{dfn}

	% We call an element $D$ of $\Hom_{\BC_M}(\CF_1,\CF_2)$ a differential operator of order $\leq r$ if
	%\[
	% (D_U: \CF_1(U)\rightarrow \CF_2(U))\in \mathrm{Diff}_r(\CF_1(U),\CF_2(U))
	%\]
	%for all open subsets $U$ of $M$, where $\CF_1(U)$ and $\CF_2(U)$ are viewed as $\CO(U)$-modules, and $\CO(U)$ is viewed as a $\C$-algebra. Here we write $D=\{D_U\}_{U\textrm{ is an open subset of }M}$, and similar notation will be used without further explanation. 
	Write $\mathrm{Diff}_r(\CF_1,\CF_2)\subset \Hom_{\BC_M}(\CF_1,\CF_2)$ for the space of differential operators of order $\leq r$.
	Then we have that
	\[
	\mathrm{Diff}_0(\CF_1,\CF_2)=\mathrm{Hom}_{\CO}(\CF_1,\CF_2),
	\]
	and
	\[
	\mathrm{Diff}_{r_1}(\CF_2,\CF_3)\circ \mathrm{Diff}_{r_2}(\CF_1,\CF_2)\subset \mathrm{Diff}_{r_1+r_2}(\CF_1,\CF_2)\qquad \textrm{for all }\, r_1,r_2\in \BN,
	\]
	where $\CF_3$ is another sheaf of $\CO$-modules.
	
	\begin{dfn}\label{def:findiff}
		An element $D\in \Hom_{\BC_M}(\CF_1,\CF_2)$ is called a finite order differential operator if it belongs to $\mathrm{Diff}_r(\CF_1,\CF_2)$ for some $r\geq 0$.
		It is called a differential operator if for every $a\in M$, there is an open neighborhood $U$ of $a$ in $M$ such that $D|_U\in \Hom_{\BC_U}(\CF_1|_U,\CF_2|_U)$ is a finite order differential operator.
	\end{dfn}
	Respectively write $\mathrm{Diff}_{\mathrm{fin}}(\CF_1,\CF_2)$ and  $\mathrm{Diff}(\CF_1,\CF_2)$ for the space of finite order differential operators and the space of differential operators in $\Hom_{\BC_M}(\CF_1,\CF_2)$.

	\begin{dfn}\label{def:deriv}
		An element  $D\in \Hom_{\BC_M}(\CO,\CF_2)$ is called a derivation if
		\[
		(D_U: \CO(U)\rightarrow \CF_2(U))\in \mathrm{Der}(\CO(U),\CF_2(U))
		\]
		for every open subset $U$ of $M$. 
	\end{dfn}
	
	The space of all derivations in $\Hom_{\BC_M}(\CO,\CF_2)$ is denoted by $\mathrm{Der}(\CO,\CF_2)$.

	Write $\mathcal D\mathrm{iff}(\CF_1,\CF_2)$ for the subsheaf of $\CH\textrm{om}_{\BC_M}(\CF_1,\CF_2)$ given by
	\[U\mapsto \mathrm{Diff}(\CF_1|_U,\CF_2|_U)\quad (U\ \text{is an open subset of}\ M).
	\]
	We also define  the subpresheaf $\mathcal D\mathrm{iff}_{\mathrm{fin}}(\CF_1,\CF_2)$ of $\CH\textrm{om}_{\BC_M}(\CF_1,\CF_2)$,
	and the subsheaf $\mathcal D\textrm{er}(\CO,\CF_2)$ of $\CH\textrm{om}_{\BC_M}(\CO,\CF_2)$ similarly.
	
	For every  $f\in \CO(M)$ and $D\in \mathrm{Diff}(\CF_1,\CF_2)$, we define two differential operators $f\circ D$ and $D\circ f$ in $\mathrm{Diff}(\CF_1,\CF_2)$  by setting
	\be\label{eq:defofDf} (f\circ D)_U:\ u\mapsto f|_U(D_U(u))\quad\text{and}\quad (D\circ f)_U:\ u\mapsto D_U (f|_Uu ),
	\ee 
	where $u\in \CF_1(U)$ and $U$ is an open subset of $M$.
	Then we have two natural  $\CO(M)$-module structures on $\mathrm{Diff}(\CF_1,\CF_2)$  with
	\begin{eqnarray}
		\label{eq:leftactiononD}  &&\CO(M)\times  \mathrm{Diff}(\CF_1,\CF_2)\rightarrow  \mathrm{Diff}(\CF_1,\CF_2),\quad (f,D)\mapsto f\circ D,\\
		\label{eq:rightactiononD} &&\CO(M)\times  \mathrm{Diff}(\CF_1,\CF_2)\rightarrow  \mathrm{Diff}(\CF_1,\CF_2),\quad (f,D)\mapsto D\circ f.
	\end{eqnarray}
	Each of the  two actions mentioned above makes
	$\mathcal D\mathrm{iff}(\CF_1,\CF_2)$ a sheaf of $\CO$-modules. 
	Unless otherwise mentioned,  $\mathcal D\mathrm{iff}(\CF_1,\CF_2)$ is
	viewed as a sheaf of $\CO$-modules with the action \eqref{eq:leftactiononD}. 
	Under this action, $\mathcal D\textrm{er}(\CO,\CF_2)$ is a sheaf of  $\CO$-modules as well (as a subsheaf of $\mathcal D\mathrm{iff}(\CO,\CF_2)$).

	Let $\CF$ be a sheaf of $\CO$-modules. 
	For every open subset $U$ of $M$, write $\CF_\mathrm{c}(U)$ for the space of compactly supported sections in $\CF(U)$. 
	Obviously, $\CF_\mathrm{c}(U)$ is an $\CO(U)$-submodule of $\CF(U)$.
	The softness of $\CF$ (see Corollary \ref{softo2}) implies that
	the assignment
	\be\label{eq:compactofF} U\mapsto \CF_\mathrm{c}(U)\quad (\text{$U$ is an open subset of $M$}),\ee
	together with the extension by zero maps 
	\be\label{eq:ext} \mathrm{ext}_{U,V}:\ \CF_\mathrm{c}(V)\rightarrow \CF_\mathrm{c}(U)\quad (V\subset U\ \text{are open subsets of $M$}),\ee 
	forms  a cosheaf of $\CO$-modules on $M$ (see \cite[Chapter V, Proposition 1.6]{Br2}). We denote the resulting cosheaf by $\CF_\mathrm{c}$.
	When $\CF=\mathcal{D}\mathrm{iff}(\CF_1,\CF_2)$, we also write 
	\be\label{eq:defofdiffc}
	\mathcal{D}\mathrm{iff}_\mathrm{c}(\CF_1,\CF_2):=\CF_\mathrm{c}\quad \text{and}\quad 
	\mathrm{Diff}_\mathrm{c}(\CF_1,\CF_2):=\CF_\mathrm{c}(M).
	\ee

	Let $U$ be an open subset of $M$.
	For every formal function $f\in \CO(M)$  supported in $U$ and  every $u\in \CF(U)$, we define 
	\be \label{eq:deffu}
	f u\in \CF(M)
	\ee
	by requiring that
	\[
	(fu)|_U=f|_U\, u \quad \textrm{and}\quad (fu)|_{M\setminus \mathrm{supp}\,f}=0.
	\]

	We have a canonical  map
	\be\label{eq:diffresmap} \textrm{Diff}_r(\CF_1,\CF_2)\rightarrow \textrm{Diff}_r(\CF_1(M),\CF_2(M)), \quad D\mapsto D_M.  \ee
	The following result shows that a differential operator $D$ of finite order is determined by its action on global sections.
	Using this, we will often not distinguish between a finite order differential operator $D$  and the map $D_M$.

	\begin{prpd}\label{prop:DOiso}
		The linear map \eqref{eq:diffresmap} is bijective for every $r\in \BN$. In particular, the canonical map
		\[\mathrm{Hom}_{\CO}(\CF_1,\CF_2)\rightarrow \mathrm{Hom}_{\CO(M)}(\CF_1(M),\CF_2(M))\]
		is bijective. 
	\end{prpd}	
	\begin{proof} Let  $U$ be an open subset of $M$.
		We choose a family $\{(U_a,U'_a,f^a)\}_{a\in U}$ of triples, where
		\begin{itemize}
			\item $U_a$, $U'_a$ are open  neighborhoods  of $a$ such that $\overline{U_a}\subset U'_a$ and $\overline{U'_a}\subset{U}$;
			\item $f^a\in \CO(M)$ is a formal function   such that
			$f^a|_{\overline{U_a}}=1$ and $f^a|_{M\setminus U'_a}=0$.
		\end{itemize}

		We first prove the injectivity of \eqref{eq:diffresmap}.
		Let $D\in \textrm{Diff}_r(\CF_1,\CF_2)$ and  $u\in\CF_1(U)$.
		%Note that for every $a\in U$, $f^au$ is a global section in %$\CF_1(M)$ such that
		%$(f^au)|_{U_a}=u|_{U_a}$.
		Since $D$ commutes with the restriction maps in $\CF_1$ and $\CF_2$,
		we have that
		\[D_U(u)|_{U_a}=D_{U_a}(u|_{U_a})=D_{U_a}((f^au)|_{U_a})=D_M(f^au)|_{U_a}.\]
		This implies that $D=0$ provided that $D_M=0$, as desired.
		
		Now we turn to prove that the map \eqref{eq:diffresmap} is surjective.  Fix a differential operator
		$P\in \textrm{Diff}_r(\CF_1(M),\CF_2(M))$.
		We claim that
		\begin{equation}\label{claimind}
			\text{for every $u'\in\CF_1(M)$ satisfying $u'|_U=0$, one has  that $P(u')|_{U}=0$.}
		\end{equation}
		Indeed, let $g\in\CO(M)$ be a formal  function such that
		\[g|_{\overline{U_a}}=0\quad\text{and}\quad g|_{M\setminus U'_a}=1.\]
		Then we have that $gu'=u'$ in $\CF_1(M)$.
		Recall that $[P,g]=0$ if  $r=0$, and $[P,g]\in \mathrm{Diff}_{r-1}(\CF_1,\CF_2)$ otherwise. Since   \[P(u')|_{U_a}=P(gu')|_{U_a}=gP(u')|_{U_a}+[P,g](u')|_{U_a}=[P,g](u')|_{U_a},\]
		it follows that $P(u')|_{U_a}=0$ by induction on $r$.
		This proves the claim \eqref{claimind} as $\{U_a\}_{a\in U}$ forms an open cover of $U$.
		
		Let $u\in \CF_1(U)$. Define 	
		\begin{eqnarray}\label{defPu}P_{U_a}(u|_{U_a}):=P(f^au)|_{U_a}\in \CF_2(U_a).\end{eqnarray}
		For  all  $a,b\in U$, we have that 
		\[(f^au)|_{U_a\cap U_b}=u|_{U_a\cap U_b}=(f^bu)|_{U_a\cap U_b}.\]
		Thus, by applying
		the claim \eqref{claimind}  (replace  $U$ by  $U_a\cap U_b$) , we have that
		\be\label{defPuind}
		P_{U_a}(u|_{U_a})|_{U_a\cap U_b}=P(f^au)|_{U_a\cap U_b}=P(f^bu)|_{U_a\cap U_b}=P_{U_b}(u|_{U_b})|_{U_a\cap U_b}.\ee
		This yields a map
		\[P_U:\CF_1(U)\rightarrow \CF_2(U),\quad u\mapsto P_U(u),\]
		where $P_U(u)$ is the section in $\CF_2(U)$ obtained by gluing the family $\{P_{U_a}(u|_{U_a})\}_{a\in U}$ together.
		
		We remark that the definition of $P_U$ is independent of the choice of the family $\{(U_a,U_a',f^a)\}_{a\in U}$.
		Indeed, let $\{(\tilde{U}_a,\tilde{U}_a',\tilde{f}^a)\}_{a\in U}$ be another such family, and as in \eqref{defPu} we define a map $\tilde{P}_U$ associated to this
		family. Then by a similar argument as  in \eqref{defPuind}, we have that
		\[P_U(u)|_{U_a\cap \tilde{U}_a}=\tilde{P}_U(u)|_{U_a\cap \tilde{U}_a},\] as desired.
		Using this independence, one concludes that the family \[D=\{P_U\}_{U \text{ is an open subset of } M}\]forms a homomorphism between the sheaves $\CF_1$ and $\CF_2$.

		Let $i\in \BN$. We now prove by induction on $i$ that  
		\begin{eqnarray}\label{eq:P_U=Pf^a}
			&&[\cdots[[P_U,h_1],h_2],\cdots,h_i](u)|_{U_a}=[\cdots[[P, f^ah_1],f^ah_2],\cdots,f^ah_i](f^au)|_{U_a}
		\end{eqnarray}
		for all $h_1,h_2,\dots,h_i\in \CO(U)$ and $u\in \CF_1(U)$. 
		When $i=0$, \eqref{eq:P_U=Pf^a} follows from the definition of $P_U$. Assume that $i>0$ and \eqref{eq:P_U=Pf^a} holds when $i$ is smaller.
		Set \[
		P_U':= [\cdots[[P_U,h_1],h_2],\cdots,h_{i-1}]\ \text{and}\  P':= [\cdots[[P,f^ah_1],f^ah_2],\cdots,f^ah_{i-1}].
		\] 
		Then  we have that 
		\begin{eqnarray*}
			&&[\cdots[[P_U,h_1],h_2],\cdots,h_i](u)|_{U_a}\\
			&=&P_U'(h_iu)|_{U_a}-(h_iP_U'(u))|_{U_a}\\
			&=&P'(f^a(h_iu))|_{U_a}-(h_i|_{U_a})(P'(f^au)|_{U_a})\\
			&=&
			(P'(f^ah_if^au))|_{U_a}-((f^ah_i)P'(f^au))|_{U_a}\\
			&=&([P', f^ah_i](f^au))|_{U_a},
		\end{eqnarray*}  
		where we have used the induction hypothesis in the second equality. This finishes the proof of \eqref{eq:P_U=Pf^a}. 
		Since $P\in \textrm{Diff}_r(\CF_1(M),\CF_2(M))$, 
		it follows from  the equality \eqref{eq:P_U=Pf^a} (with $i=r+1$) that  $P_U\in \textrm{Diff}_r(\CF_1(U),\CF_2(U))$. 
		Then  $D\in \textrm{Diff}_r(\CF_1,\CF_2)$
		and hence the map \eqref{eq:diffresmap} is surjective as  $D_M=P$. 
		
		In conclusion, we have proved the first assertion of the proposition. The second assertion follows by taking $r=0$ in the first assertion. 
	\end{proof}

	%Note that the assignment
	%\[U\mapsto\Hom_{\BC_U}(\CF_1|_U,\CF_2|_U)\quad\ \text{($U$ are open subsets of $M$)}\]
	% determines a sheaf, denoted as $\CH\textrm{om}_{\BC_M}(\CF_1,\CF_2)$, of $\BC$-vector spaces over $M$.
	%Write $\mathcal D\mathrm{iff}_r(\CF_1,\CF_2)$ for the subsheaf of $\CH\mathrm{om}(\CF_1,\CF_2)$ given by $U\mapsto \mathrm{Diff}_r(\CF_1|_U,\CF_2|_U)$.
	%Similarly, we define the subsheaf $\mathcal{D}\mathrm{iff}(\CF_1,\CF_2)$ of $\CH\textrm{om}_{\BC_M}(\CF_1,\CF_2)$, the sub-presheaf $\mathcal D\mathrm{iff}_{\mathrm{fin}}(\CF_1,\CF_2)$ of $\CH\textrm{om}_{\BC_M}(\CF_1,\CF_2)$, and the subsheaf $\mathcal D\textrm{er}(\CO,\CF_2)$ of $\CH\textrm{om}_{\BC_M}(\CO,\CF_2)$.

	By Proposition \ref{prop:DOiso}, we have 
	an identification $\mathrm{Diff}_0(\CO,\CF_2)=\CF_2(M)$. Similar to \eqref{diff1}, we obtain that
	\be\label{eq:diff=der+1}
	\mathrm{Diff}_1(\CO,\CF_2)=\mathrm{Der}(\CO,\CF_2)\oplus \CF_2(M).
	\ee

	Let $\CF$ be a sheaf of $\CO$-modules and let  $D\in\mathrm{Diff}_r(\CO,\CF)$.
	For every $a\in M$,  write
	\be\label{l}
	D_a: \CO_a\rightarrow \CF_a
	\ee
	for the linear map on the stalks induced by $D$.
	It is clear that $D_a\in \mathrm{Diff}_r( \CO_a,\CF_a)$.
	Furthermore, Lemma \ref{dlocal0} implies that
	$D_a$ induces a linear map
	\be\label{dlocal}
	\wh{D}_a: \widehat \CO_a\rightarrow \widehat \CF_a.
	\ee
	%The following result asserts that $\widehat{D}_a$ is a differential operator  by regarding  $\widehat\CF_a$ as an $\widehat\CO_a$-module as well.
	
	\begin{lemd}\label{lem:hatD}
		We have that   $\wh{D}_a\in \mathrm{Diff}_r(\widehat \CO_a,\widehat{\CF}_a)$.
	\end{lemd}
	\begin{proof}
		Note that 
		\[
		([\cdots[[\widehat{D}_a,f_0],f_1],\cdots, f_r])(f_{r+1}) =0
		\]
		for all $f_0,f_1, \dots,f_r, f_{r+1}$ in the image of the map $\CO_a\rightarrow \widehat \CO_a$. The second assertion then follows since the map
		$\CO_a\rightarrow \widehat \CO_a$ is surjective (see  Proposition
		\ref{lem:formalstalkiso}).
	\end{proof}

	Let $n,k\in \BN$, and let $N$ be   an open submanifold of $\R^n$.
	Let $\CF$ be a sheaf of $\CO_N^{(k)}$-modules, and let $\{f_{I,J}\}_{I\in \BN^n, J\in \BN^k}$ be a locally finite family of elements of $\CF(N)$.
	Identifying $\mathrm{Diff}_0(\CO_N^{(k)},\CF)$ with $\CF(N)$, %(see Proposition \ref{prop:DOiso}),
	we see that
	\be\label{do0}
	\sum_{I\in \BN^n, J\in \BN^k} f_{I,J}\circ \partial_x^{I}\partial_y^J
	\ee
	is an element of $\mathrm{Diff}(\CO_N^{(k)},\CF)$.
	Conversely, we have the following result.
	
	\begin{prpd}\label{dofm0}
		Assume that $M=N^{(k)}$ and the sheaf $\CF$ of $\CO$-modules satisfies the following  conditions:
		\begin{itemize}\label{cdt}
			\item
			the canonical map $\CF(M)\rightarrow  \prod_{a\in M} \widehat \CF_a $ is injective; and
			\item
			$ \cap_{i\in\BN}\widehat{\m}_a^i\widehat{\CF}_a=0$ for every $a\in M$, where
			$\widehat{\m}_a$ is the maximal ideal of $\widehat{\CO}_a$.			\end{itemize}
		Then every element of $\mathrm{Diff}(\CO,\CF)$ is uniquely of the form \eqref{do0}.
	\end{prpd}

	\begin{proof} Without loss of generality, we assume that $M\neq\emptyset$.
		Write  \[\BC[\mathbf{x,y}]:=\BC[x_1, x_2, \dots, x_n, y_1, y_2, \dots, y_k]\] for simplicity.
		We view $\BC[\mathbf{x,y}]$ as a subspace of $\CO(M)$, as well as a subspace of $\widehat \CO_a$ for each $a\in M$
		(see \eqref{eq:formalstalkiso}).
		Then the uniqueness assertion follows by applying the differential operator  \eqref{do0} to the monomials in $\BC[\mathbf{x,y}]$.
		
		Write
		$\BC[\mathbf{x,y}]_{\le r}\subset \BC[\mathbf{x,y}]$  for the subspace consisting of polynomials of degree $\leq r$.
		We claim that the map
		\[
		\begin{array}{rcl}
			\Psi_r: \mathrm{Diff}_r(\CO,\CF)&\rightarrow& \Hom_\BC(\BC[\mathbf{x,y}]_{\le r}, \CF(M)),\\
			D&\mapsto&(f\mapsto D_M(f))
		\end{array}
		\]
		is injective.
		Indeed, take a differential operator $D\in  \mathrm{Diff}_r(\CO,\CF)$  such that $D_M(f)=0$ for all $f\in \BC[\mathbf{x,y}]_{\le r}$. For every $a\in M$,
		Lemma \ref{lem:hatD} implies that the map
		$
		\widehat{D}_a:  \widehat \CO_a\rightarrow \widehat \CF_a
		$
		belongs to $\mathrm{Diff}_r( \widehat \CO_a,\widehat{\CF}_a)$ and vanishes on $\BC[\mathbf{x,y}]_{\le r}$.
		It then follows from Lemma \ref{diff} and the second condition in the proposition that $\widehat{D}_a$ is the zero map. By
		%In view of (problem byW:the comma of diagram)
		the commutative diagram
		\[\xymatrix{
			\CO(M) \ar[d] \ar[r]^{D_M} & \CF(M) \ar[d] \\
			\wh{\CO}_a \ar[r]^{\widehat{D}_a} & \wh{\CF}_a,}\]
		we conclude  that for every $g\in \CO(M)$, the image of $D_M(g)$ in $\wh{\CF}_a$ is zero for all $a\in M$.
		Then the first condition in the proposition forces that $D_M=0$, which proves the claim by using Proposition \ref{prop:DOiso}.
		
		Write $\mathrm{Diff}'_r(\CO,\CF)\subset \mathrm{Diff}_r(\CO,\CF)$ for the space of all differential operators of the form \eqref{do0} with $f_{I,J}=0$ whenever $\abs{I}+\abs{J}> r$.
		It is easy to see that the map ${\Psi_r}|_{\mathrm{Diff}'_r(\CO,\CF)}$ is surjective. Hence $\mathrm{Diff}'_r(\CO,\CF)= \mathrm{Diff}_r(\CO,\CF)$. Together with the uniqueness assertion, this implies that every differential operator in $\mathrm{Diff}(\CO,\CF)$ has the form  \eqref{do0}.
		
	\end{proof}
	
	\begin{exampled}\label{ex:diff}
		For every $l\in \BN$, it follows from Example \ref{ex:ol} (see \eqref{eq:formalstalkOnl}) that the sheaf $\CO_N^{(l)}$ of $\CO_N^{(k)}$-modules 
		satisfies the conditions in Proposition \ref{dofm0}.
		Then every element of $\mathrm{Diff}(\CO_{N}^{(k)},\CO_{N}^{(l)})$ is uniquely of the form
		\[ \sum_{I\in \BN^n, J\in \BN^k} f_{I,J}\circ \partial_x^I\partial_y^J,\]
		where $\{f_{I,J}\}_{I\in \BN^n, J\in\BN^k}$ is a locally finite family of elements of $\CO_{N}^{(l)}(N)$.
	\end{exampled}

	As a by-product of Proposition \ref{dofm0}, we have the following result.

	\begin{cord}\label{cor:unders} As an $\CO_N^{(k)}(N)$-module,  $\mathrm{Der}(\CO_N^{(k)},\CO_N^{(k)})$  is free of rank $n+k$, with a set of free generators
		given by
		\[
		\{\partial_{x_1}, \partial_{x_2}, \dots, \partial_{x_n},\, \partial_{y_1}, \partial_{y_2}, \dots, \partial_{y_k}\}.
		\]
	\end{cord}
	\begin{proof}
		With the notations as in the proof of Proposition \ref{dofm0}, we have shown  that
		\[
		\mathrm{Diff}'_1(\CO_N^{(k)},\CO_N^{(k)})=\mathrm{Diff}_1(\CO_N^{(k)},\CO_N^{(k)}).
		\]
		The assertion then follows by using \eqref{eq:diff=der+1} and the fact that 
		\[
		\mathrm{Der}(\CO_N^{(k)},\CO_N^{(k)})=\{D\in \mathrm{Diff}_1(\CO_N^{(k)},\CO_N^{(k)})\mid D(1)=0\}.
		\]
	\end{proof}

	\section{The topology on the space of formal functions}\label{sec:smoothtop}

	In this section, we introduce a canonical topology on the space of formal functions on $M$.

	\subsection{Smooth topology on $\CF(M)$}\label{topology}
	Recall the reduction $\underline{\CO}$ of $\mathcal{O}$ defined in \eqref{underco}, to be viewed as a sheaf of $\CO$-modules by using the reduction map.
	Let $\CF$ be a sheaf of $\CO$-modules and let $D\in \mathrm{Diff}_\mathrm{c}(\CF,\underline \CO)$. For every $u\in \CF(M)$, define
	\[
	\abs{u}_D:=\sup_{a\in M} \abs{(D(u))(a)}.
	\]
	Then $\abs{\,\cdot\,}_D$ is a seminorm on $\CF(M)$. 
	
	\begin{dfn}\label{de:smoothtopology}
		The locally convex  topology on $\CF(M)$ defined by the family 
		\[\{\abs{\,\cdot\,}_D\}_{D\in \mathrm{Diff}_\mathrm{c}(\CF,\underline \CO)}\] of 
		seminorms is called the smooth topology of $\CF(M)$.
	\end{dfn}
	
	% Under the smooth topology, $\CF(M)$ is naturally an LCS.
	Unless otherwise mentioned, $\CF(M)$ is equipped with the smooth topology.

	\begin{exampled}\label{topm0}
		Assume that $M=N^{(k)}$, where $N$ is a smooth manifold and $k\in \BN$.
		Then it follows from  Example \ref{ex:diff} that  the smooth topology on $\CO(M)=\RC^{\infty}(N)[[y_1, y_2, \dots, y_k]]$
		coincides with the term-wise  convergence topology, where $\RC^{\infty}(N)$ is equipped with the usual smooth topology. \end{exampled}
	
	\begin{lemd}\label{lemd:O(M)-moduleO(F)}%{wadd}
		For every $f\in \CO(M)$,
		the linear map \[\CF(M)\rightarrow \CF(M),\quad u\mapsto{fu}\] is continuous.
	\end{lemd}
	\begin{proof} 
		For every $D\in \mathrm{Diff}_\mathrm{c}(\CF,\underline{\CO})$, we have that \[\abs{fu}_{D}=\abs{u}_{D\circ f}\]
		for all $u\in \CF(M)$, where $D\circ f\in \mathrm{Diff}_\mathrm{c}(\CF,\underline{\CO})$ is as in \eqref{eq:defofDf}. This implies the lemma.
	\end{proof}
	
	\begin{lemd}\label{lem:HausdorffFM}
		The smooth topology on $\CF(M)$ is Hausdorff if and only if for every $0\ne u\in \CF(M)$, there exists some $D\in \mathrm{Diff}_\mathrm{c}(\CF,\underline \CO)$
		such that $D(u)\ne 0$.
		Furthermore, if $\CF(M)$ is Hausdorff, then so is $\CF(U)$ for every open subset $U$  of $M$.
	\end{lemd}
	\begin{proof} The first assertion is obvious. For the second one, let $0\ne u\in \CF(U)$.
		Take  open subsets $V$ and $W$ of $U$ such that
		\[
		V\subset \overline{V}\subset W\subset\overline{W}\subset U\quad 
		\text{and}\quad u|_V\ne 0.
		\]
		And take a formal function $f$ on $M$
		such that $\mathrm{supp}\,f\subset W$ and $f|_{\overline{V}}=1$.
		Since 
		\[(f u)|_V=u|_V\ne 0 \quad \text{and} \quad  \mathrm{supp}(fu)\subset W,\] together with the condition that $\CF(M)$ is Hausdorff,  there is a differential operator  $D'\in \mathrm{Diff}_\mathrm{c}(\CF,\underline \CO)$ such that
		$D'(fu)|_W\ne 0$. Set \[D:=(D'\circ f)|_U\in \mathrm{Diff}_\mathrm{c}(\CF|_U,\underline \CO|_U).\] Then $D(u)\ne 0$ as 
		\[D(u)|_W=((D'\circ f)|_U(u))|_W=D'(fu)|_W\ne 0.\]
		This finishes the proof.
	\end{proof}
	%The first assertion is obvious. For the second one, let $0\ne u\in \CF(U)$.Take  relatively compact open subsets $V$ and $W$ of $U$ such that
	%\[
	%V\subset \overline{V}\subset W\subset\overline{W}\subset U\quad 
	%\text{and}\quad u|_V\ne 0.
	%\]
	%And, take a formal function $f$ on $M$such that $\mathrm{supp}\,f\subset W$ and $f|_{\overline{V}}=1$.Since $\CF(M)$ is Hausdorff and $(f u|_W)|_V=u|_V\ne 0$, there is a $D'\in \mathrm{Diff}_\mathrm{c}(\CF,\underline \CO)$ such that$D'(fu|_W)\ne 0$. Set $D:=(D'f)|_U\in\mathrm{Diff}_\mathrm{c}(\CF|_U,\underline \CO|_U)$. Then $D(u)\ne 0$ as 
	%\[(D(u))|_W=(D'f)|_W(u|_W)=D'(fu|_W)\ne 0.\]This finishes the proof.

	\begin{lemd}\label{lemr0}
		Let $U$  be an open subset of $M$. Then the restriction map
		\be\label{redo0}
		\CF(M)\rightarrow \CF(U)
		\ee
		is %linear and 
		continuous.
		
	\end{lemd}
	
	\begin{proof}
		For every $D\in \mathrm{Diff}_\mathrm{c}(\CF|_U,\underline{\CO}|_U)$, we have that
		\[
		\abs{u}_{\mathrm{ext}_{M,U}(D)}= \left|(u|_U)\right|_{D}\quad \textrm{for all $u\in \CF(M)$}.
		\]
		This implies that the restriction map \eqref{redo0} is continuous.
	\end{proof}
	
	\begin{lemd}\label{lemr1}
		Let $\{U_\gamma\}_{\gamma\in \Gamma}$ be an open cover of $M$.  Then the linear map
		\be\label{redo}
		\CF(M)\rightarrow \prod_{\gamma\in \Gamma} \CF(U_\gamma), \quad u\mapsto \{u|_{U_\gamma}\}_{\gamma\in \Gamma}
		\ee
		is a topological embedding of topological spaces.
		Furthermore, if $\CF(M)$ is Hausdorff,
		then the image of \eqref{redo} is closed.
	\end{lemd}
	\begin{proof} The sheaf property of $\CF$ implies  that the map \eqref{redo} is injective.
		Lemma \ref{lemr0} implies that the map \eqref{redo} is continuous.
		Let $D\in \mathrm{Diff}_\mathrm{c}(\CF,\underline \CO)$. Choose  $\gamma_1, \gamma_2, \dots, \gamma_s\in \Gamma$ ($s\in\BN$) such that
		\[\mathrm{supp}\,D\subset U_{\gamma_1}\cup U_{\gamma_2}\cup \cdots \cup U_{\gamma_s}.\]
		Let $\{g, f_{\gamma_1}, f_{\gamma_2}, \dots, f_{\gamma_s}\}$ be a partition of unity on $M$ subordinate to the open cover
		$\{V:=M\setminus \mathrm{supp}\,D, U_{\gamma_1},U_{\gamma_2},\dots, U_{\gamma_s}\}$.
		Note that for each $i=1,2,\dots, s$, we have that 
		\[(f_{\gamma_i}\circ D)|_{U_{\gamma_i}}\in \mathrm{Diff}_\mathrm{c}(\CF|_{U_{\gamma_i}},\underline{\CO}|_{U_{\gamma_i}}),\]
		where $f_{\gamma_i}\circ D\in \mathrm{Diff}_\mathrm{c}(\CF,\underline{\CO})$ is as in \eqref{eq:defofDf}.  
		For every $u\in \CF(M)$, we obtain that
		\[
		\abs{u}_D\leq \abs{u}_{g\circ D}+ \sum_{i=1}^s \abs{u}_{f_{\gamma_i}\circ D} =\sum_{i=1}^s \left|(u|_{U_{\gamma_i}})\right|_{(f_{\gamma_i}\circ D)|_{U_{\gamma_i}}}.
		\]
		Therefore the continuous linear map \eqref{redo} is a homeomorphism onto its image.

		Furthermore, if $\CF(M)$ is Hausdorff, then by Lemma \ref{lem:HausdorffFM},  $\CF(U_\gamma\cap U_{\gamma'})$ is  also Hausdorff for all $\gamma, \gamma'\in \Gamma$.
		One  concludes from this and the sheaf property of $\CF$  that the map \eqref{redo}
		has closed image. 
	\end{proof}

	\begin{cord} \label{cor:CFMisNF}
		Assume that the sheaf $\CF$ of $\CO$-modules is locally free of finite rank. Then the space $\CF(M)$ is a product of nuclear Fr\'echet spaces. 
	\end{cord}
	\begin{proof}
		First assume that $M$ is secondly countable. Take an atlas $\{U_i\}_{i\in\mathbb{N}}$ of $M$ such that each $\CF(U_i)$  is free of finite rank. Since each $\CO(U_i)$ is a nuclear  Fr\'echet space (by Example \ref{topm0}), so is $\CF(U_i)$. 
		Using  Lemma \ref{lemr1}, one concludes that  $\CF(M)$ is Hausdorff and moreover, $\CF(M)$ is a nuclear Fr\'echet space.
		
		For the general case, the assertion follows from the fact that
		\be\label{eq:FM=prod}
		\CF(M)=\prod_{Z\in \pi_0(M)} \CF(Z)
		\ee
		as LCS, which is again implied by Lemma \ref{lemr1}.
	\end{proof}

	By applying  Lemma 
	\ref{lemr1}, Example \ref{topm0} 
	and Corollary \ref{cor:CFMisNF}, one immediately gets the following result. 
	
	\begin{prpd}\label{prop:OMnuclear}
		The $\BC$-algebra  $\CO(M)$ is a topological $\BC$-algebra. As a topological vector space, it is a product of nuclear Fr\'echet spaces. Especially, the LCS $\CO(M)$ is complete and nuclear.
	\end{prpd}

	\subsection{Morphisms to $\R^m$}
	Define an $\R$-subalgebra
	\[
	\CO(M;\R):=\{f\in \CO(M)\mid f(a)\in \R\textrm{ for all }a\in M\}
	\]
	of $\CO(M)$.
	Let $m\in \BN$, and let
	\[
	\varphi=(\overline\varphi, \varphi^*): (M,\CO)\rightarrow (\R^m, \CO_{\R^m}^{(0)})
	\]
	be a morphism of formal manifolds. Recall that $x_1, x_2, \dots , x_m\in  \CO_{\R^m}^{(0)}(\R^m)$ are standard coordinate functions of $\R^m$.
	Then $\varphi$ yields an $m$-tuple
	\[
	c_\varphi:=(\varphi^*(x_1), \varphi^*(x_2), \dots, \varphi^*(x_m))\in (\CO(M;\R))^m.
	\]

	The main goal of this subsection is to prove the following theorem.
	\begin{thmd}\label{thmmapr}
		The map
		\be\label{mapcp}
		\{\text{morphism from $(M,\CO)$ to $(\R^m,\CO_{\R^m}^{(0)})$}\}\rightarrow (\CO(M;\R))^m, \quad \varphi\mapsto c_\varphi
		\ee
		is bijective.
	\end{thmd}
	
	We will prove the theorem through a sequence of lemmas.
	Firstly, the continuous map $\overline\varphi: M\rightarrow \R^m$ is determined by $c_\varphi$ by the following lemma.
	
	\begin{lemd}\label{lem:cvarphi1}
		Let $a\in M$. Then
		\[
		\overline\varphi(a)=((\varphi^*(x_1))(a), (\varphi^*(x_2))(a), \dots, (\varphi^*(x_m))(a)).
		\]
	\end{lemd}
	\begin{proof}
		Write $b:=\overline\varphi(a)$. For every $i=1,2,\dots,m$, we have that
		\[
		(\varphi^*(x_i))(a)=x_i(b).
		\]
		This implies that $b=((\varphi^*(x_1))(a), (\varphi^*(x_2))(a), \dots, (\varphi^*(x_m))(a))$.
		
	\end{proof}

	\begin{lemd}\label{injphi}
		The map \eqref{mapcp} is injective.
	\end{lemd}
	\begin{proof}  For simplicity, let us set $\CO':=\CO_{\R^m}^{(0)}$.
		Let
		\[
		\varphi=(\overline\varphi, \varphi^*), \psi=(\overline\psi, \psi^*) : (M, \CO)\rightarrow (\R^m, \CO')
		\]
		be two morphisms such that $c_\varphi=c_\psi$.
		From Lemma \ref{lem:cvarphi1}, it follows that $\overline\varphi=\overline\psi$.
		Let $a\in M$ and write $b:=\overline\varphi(a)=\overline\psi(a)$. Then the maps
		\[
		\varphi^*_a, \psi^*_a:\ \widehat{\CO'}_b\rightarrow \widehat \CO_a
		\]
		are equal, since they are both continuous (see Lemma \ref{lem:homonformalstalkscon}) and agree on the dense subalgebra generated by the images of the coordinate functions $x_1, x_2, \dots, x_m$ under the canonical map \[\CO'(\R^m)\rightarrow \widehat{\CO'}_b.\]
		Then the lemma follows from Lemma \ref{homfst}.
	\end{proof}

	\begin{lemd}\label{morr}
		If $M=N^{(k)}$ with $N$ an open submanifold of $\R^n$ and $n,k\in \BN$, then the map
		\eqref{mapcp} is surjective.
	\end{lemd}
	\begin{proof} Let $(f_1,f_2,\dots,f_m)$  be an $m$-tuple in $(\CO(M;\R))^m$.
		It  yields a  smooth map
		\[
		\overline\varphi: N\rightarrow \R^m, \quad a\mapsto (f_1(a), f_2(a), \dots, f_m(a)).
		\]
		Let $U$ be an open subset of $\R^m$ and set  $V:=\overline\varphi^{-1}(U)\subset N$.
		For every smooth function $g\in \RC^\infty(U)$,  define a formal function
		\be\label{eq:varphiexpression} \varphi_U^*(g):=\sum_{I\in \BN^m} \left(\frac{\partial_x^I(g)}{I!}\circ \overline\varphi\right)\cdot f^I|_V\in \RC^\infty(V)[[y_1,y_2,\dots,y_k]],\ee
		where 
		\[f^I:=(f_1- \underline{f_1})^{i_1}(f_2-\underline{f_2})^{i_2}\cdots(f_m-\underline{f_m})^{i_m}\in \RC^\infty(N)[[y_1, y_2, \dots, y_k]],\] for $I=(i_1,i_2,\dots,i_m)\in \BN^m$.
		Here $\underline{f_j}\in \RC^\infty(N)\subset \RC^\infty(N)[[y_1, y_2, \dots, y_k]] $ ($j=1,2,\dots, m$) is the reduction of $f_j$ as defined in \eqref{underf}.
		
		By using Leibniz's rule, one easily checks   that the map
		\[
		\varphi^*_U: \RC^{\infty}(U)\rightarrow \RC^\infty(V)[[y_1, y_2, \dots, y_k]], \quad g\mapsto \varphi^*_U(g)
		\]
		is an algebra homomorphism. %Furthermore,  if $g(\overline\varphi(a))=0$ for some $a\in V$ and $g\in \RC^\infty(U)$, then $(\varphi^*(g))(a)=0$.
		Then we obtain a morphism
		\[
		\varphi:=(\overline \varphi, \{\varphi^*_U\}_{U\textrm{ is an open subset of $\R^m$}}): (N, \CO_{N}^{(k)})\rightarrow (\R^m, \CO_{\R^m}^{(0)})
		\]
		of formal manifolds. Note that for every $i=1,2,\dots,m$,
		\[\varphi^*(x_i)= x_i\circ \overline{\varphi}+(f_i-\underline{f_{i}})=f_i.\]
		This implies that $c_{\varphi}=(f_1, f_2, \dots, f_m)$, and  the proof is finished.
		
	\end{proof}

	By %In view of 
	Lemma \ref{injphi}, we only need to show the following lemma to finish the proof of Theorem \ref{thmmapr}.
	
	\begin{lemd}\label{morr2}
		The map
		\eqref{mapcp} is surjective.
	\end{lemd}
	\begin{proof}
		Let $f_1, f_2,\dots, f_m\in \CO(M;\R)$, and take an atlas  $\{U_\gamma\}_{\gamma\in \Gamma}$ of $M$.
		Then for each $\gamma\in \Gamma$, it follows from Lemma \ref{morr}  that there is a morphism
		\[\varphi_\gamma: (U_\gamma, \CO|_{U_\gamma})\rightarrow (\R^m, \CO_{\R^m}^{(0)})\] such that
		\[
		c_{\varphi_\gamma}=((f_1)|_{U_\gamma}, (f_2)|_{U_\gamma}, \dots,(f_m)|_{U_\gamma}).
		\]
		Moreover, Lemma \ref{injphi} implies that $\varphi_\gamma$ and $\varphi_{\gamma'}$
		agree on $(U_\gamma\cap U_{\gamma'}, \CO|_{U_\gamma\cap U_{\gamma'}})$ for all $\gamma, \gamma'\in \Gamma$.
		Hence the family $\{\varphi_\gamma\}_{\gamma\in \Gamma}$ patches together to
		a morphism \[\varphi: (M, \CO)\rightarrow (\R^m, \CO_{\R^m}^{(0)}).\]
		It is clear that $c_\varphi=(f_1, f_2,\dots, f_m)$. This proves the lemma.
		
	\end{proof}
	%Recall that $\Spec (\BC)$ is a locally ringed space whose underlying topological space is a singleton and whose ring of global sections is $\BC$. A $\BC$-locally ringed space is defined to be a locally ringed space $M$ together with a morphism $M\rightarrow \Spec (\BC)$ of locally ringed spaces.

	\subsection{Automatic continuity}
	In this subsection, we prove that for every  morphism  of formal manifolds, the induced homomorphism on the topological $\BC$-algebras  of formal functions
	is continuous.
	We start with  the following special case.

	\begin{lemd}\label{morr3}Let $n,k,m,l\in\BN$.
		For every morphism
		$\varphi=(\overline\varphi, \varphi^*): (\R^n, \CO_{\R^n}^{(k)})\rightarrow (\R^m, \CO_{\R^m}^{(l)})$, the homomorphism
		\[
		\varphi^*: \RC^\infty(\R^m)[[y_1,y_2, \dots, y_l]]\rightarrow \RC^\infty(\R^n)[[y_1,y_2, \dots, y_k]]
		\]
		is continuous.
	\end{lemd}
	\begin{proof}
		When $l=0$, this follows from Theorem \ref{thmmapr} and the explicit expression of $\varphi^*$ given in  \eqref{eq:varphiexpression}. In general, we have an obvious morphism
		\[
		\psi=(\overline\psi, \psi^*): (\R^m, \CO_{\R^m}^{(l)})\rightarrow (\R^{m+l}, \CO_{\R^{m+l}}^{(0)})
		\]
		so that the map
		\be\label{borel1}
		\psi^*:  \RC^\infty(\R^{m+l})\rightarrow \RC^\infty(\R^m)[[y_1,y_2, \dots, y_l]]
		\ee
		is given by the Taylor series expansion along $\R^l$.
		It follows from Borel's lemma that the map $\psi^*$ is surjective.
		Thus by the open mapping theorem,  it is also open.

		Consider the maps
		\[
		\RC^\infty(\R^{m+l})\xrightarrow{\psi^*} \RC^\infty(\R^m)[[y_1,y_2, \dots, y_l]]\xrightarrow{\varphi^*} \RC^\infty(\R^n)[[y_1,y_2, \dots, y_k]].
		\]
		We have shown that the composition of these two arrows is continuous, and the first arrow is open and surjective. Thus the second arrow is also continuous.
		
	\end{proof}

	\begin{thmd}\label{thm:autocon}
		Let $\varphi=(\overline\varphi, \varphi^*): (M,\CO)\rightarrow (M', \CO')$ be a morphism of formal manifolds. Then the $\BC$-algebra homomorphism
		\[
		\varphi^*: \CO'(M')\rightarrow \CO(M)
		\]
		is continuous.
	\end{thmd}

	\begin{proof}
		By %In view of 
		Lemma \ref{lemr1}, we only need to show that for every $a\in M$, there is an open neighborhood $U$ of it in  $M$ such that the homomorphism
		\be\label{varphistar1}
		\varphi^*_{M',U}: \CO'(M')\rightarrow \CO(U)
		\ee
		is continuous.
		Choose an open neighborhood  $U$ of $a$ in $M$, and an open neighborhood  $U'$ of $\overline{\varphi}(a)$ in $M'$ such that
		\begin{itemize}
			\item
			$(U, \CO|_U)\cong(\R^n, \CO_{\R^n}^{(k)})$ ($n:=\dim_a(M)$, $k:=\deg_a(M)$);
			\item
			$(U', \CO'|_{U'})\cong(\R^m, \CO_{\R^m}^{(l)})$ ($m:=\dim_{\overline{\varphi}(a)}(M')$, $l:=\deg_{\overline{\varphi}(a)}(M')$); and
			\item
			$\overline\varphi(U)\subset U'$.
		\end{itemize}
		
		View $\varphi^*_{M',U}$ as the  composition of the following two maps:
		\[
		\CO'(M')\xrightarrow{\textrm{restriction}} \CO'(U')\xrightarrow{\varphi^*_{U',U}} \CO(U).
		\]
		Note that the first arrow is continuous by Lemma \ref{lemr0}, and the second one is continuous by Lemma \ref{morr3}. Hence the  map
		$\varphi^*_{M',U}$ is  also continuous and the theorem is proved.
		
	\end{proof}
	
	\begin{cord}\label{cor:autoopen}
		Let $\varphi=(\overline\varphi, \varphi^*): (M,\CO)\rightarrow (M', \CO')$ be a morphism of formal manifolds. If the homomorphism
		$
		\varphi^*: \CO'(M')\rightarrow \CO(M)
		$
		is surjective, then it is open. 
	\end{cord}
	\begin{proof}
		In view of \eqref{eq:FM=prod}, it suffices to prove the case that $M'$ is connected. Note that $\CO'(M')$ is a Fr\'echet space, and  $\CO(M)$ is a product of Fr\'echet spaces (see Corollary \ref{cor:CFMisNF}). 
		Together with the  facts that 
		\begin{itemize}
			\item all Fr\'echet spaces are barreled (see \cite[\S\,21.5 (3)]{Ko1}); and 
			\item the product of arbitrary barreled spaces is barreled (see \cite[\S\,27.1 (5)]{Ko1}),
		\end{itemize}
		it implies that $\CO(M)$ is a barreled space. The assertion then follows from Theorem \ref{thm:autocon} and the fact that every continuous surjective linear map from a Fr\'echet space to a barreled space is open (see \cite[\S\,34.2 (3)]{Ko2}). 
	\end{proof}

	\section{The algebra of formal functions}\label{sec:formalspec}
	In this section, we show that a formal manifold is determined by the topological $\BC$-algebra of formal functions. 
	\subsection{Characters of $\CO(M)$}
	For every $\BC$-algebra $A$, we call a  $\BC$-algebra homomorphism $A\rightarrow \BC$ a character of $A$.
	For every $a\in M$, the map
	\be\label{eq:defEva}
	\mathrm{Ev}_a: \CO(M)\rightarrow \BC, \quad f\mapsto f(a)
	\ee
	of evaluating at $a$ is a continuous character of $\CO(M)$.
	The main goal of this subsection is to prove the following theorem, which is known as Milnor's exercise when $M$ is a smooth
	manifold (see \cite[Problem 1-C]{MS}).

	\begin{thmd}\label{characterm2}
		Suppose that the cardinality of $\pi_0(M)$ does not exceed the cardinality of $\R$.
		Then every character of $\CO(M)$ is of  the form $\mathrm{Ev}_a$ for some $a\in M$.
	\end{thmd}

	Before proving the theorem, we establish some technical lemmas.
	
	\begin{lemd}\label{exf1}
		If $\pi_0(M)$ is countable, then
		there exists a formal function $f\in \CO(M;\R)$ such that  the set $\{a\in M \mid f(a)\leq C\}$ is compact for all $C\in \R$.
	\end{lemd}
	\begin{proof}
		Recall that the reduction $\underline{M}$ of $M$ is a smooth manifold.
		It is well-known that there is a bounded below proper smooth map $\underline{f}: \underline{M}\rightarrow \R$ (see the proof of \cite[Chapter 2, Theorem 10.8]{Br} for example). Let $f\in \CO(M)$ be an element in the preimage of $\underline{f}$ under the surjective map \eqref{suroo}. Then $f$ fulfills the requirement of the lemma.
	\end{proof}

	\begin{lemd}\label{threechi}
		Let $\chi: \CO(M)\rightarrow \BC$ be a character, and let $a\in M$. If $\chi\neq \mathrm{Ev}_a$, then there is an element $f_{\chi}^a\in \CO(M)$ such that
		\begin{itemize}
			\item
			$f_{\chi}^a(b)$ is a non-negative real number for all  $b\in M$;
			\item $f_{\chi}^a(a)=1$; and
			\item $\chi(f_{\chi}^a)=0$.
		\end{itemize}
		
	\end{lemd}
	\begin{proof}
		The condition $\chi\neq \mathrm{Ev}_a$ implies that there is an element $f\in \CO(M)$ such that $f(a)\neq 0$ and $\chi(f)=0$.
		Let $U$ be an open neighborhood of $a$ in $M$ such that $(U, \CO|_U)\cong (\R^n, \CO_{\R^n}^{(k)})$ for some $n,k\in \BN$.
		Multiplying $f$ by a suitable formal function if necessary, we assume without loss of generality that $\mathrm{supp}\,f$ is contained in $U$. Then we have a decomposition $f=f_1+\sqrt{-1} f_2$ such that $f_1, f_2\in \CO(M;\R)$, and the lemma follows by defining
		\[
		f_{\chi}^a:=(f_1(a)^2+f_2(a)^2)^{-1}\cdot (f_1+\sqrt{-1} f_2)\cdot (f_1-\sqrt{-1} f_2).
		\]
	\end{proof}
	
	For every $Z\in \pi_0(M)$, write $1_Z$ for the identity element of $\CO(N)$.
	\begin{lemd}\label{exf2}
		Suppose that the cardinality of $\pi_0(M)$ does not exceed the cardinality of $\R$.
		Then for every  character  $\chi$ of $\CO(M)$, there is a formal function
		$f_\chi\in \CO(M)$ such that
		\begin{itemize}
			\item
			$f_\chi(a)$ is a non-negative real number for all $a\in M$;
			\item the set
			$
			Z_0:=\{b\in M\mid f_\chi(b)=0\}
			$
			is  a connected component of $M$; and
			\item $\chi(f_\chi)=0$.
		\end{itemize}
		
	\end{lemd}
	
	\begin{proof}
		Fix an injective map
		\[
		\tau:\ \pi_0(M)\rightarrow \R.
		\]
		Let $g$ be the formal function on $M$ defined by setting $g|_Z=\tau(Z)1_Z$ for all $Z\in \pi_0(M)$,
		and set
		\[
		f_\chi:=(g-\chi(g))\cdot (g-\overline{\chi(g)})\qquad ( \ \overline{\phantom{a}}\ \textrm{ indicates the complex conjugation}).
		\]
		Then
		\[
		\chi(f_\chi)=0\quad\textrm{and}\quad f_\chi(a)\geq 0\ \textrm{ for all }a\in M.
		\]
		Moreover,
		\[
		Z_0:=\{b\in M\mid f_\chi(b)=0\}
		\]
		is either the empty set or a connected component of $M$. If $Z_0$ is empty, then $f_\chi$ is invertible in $\CO(M)$ by   Lemma \ref{lem:inverse},
		which contradicts $\chi(f_\chi)=0$. Thus $f_\chi$ is a desired formal function.
	\end{proof}

	\noindent\textbf{Proof of Theorem \ref{characterm2}:}
	Let $\chi: \CO(M)\rightarrow \BC$ be a character of $\CO(M)$.
	Let $f_\chi$ be as in Lemma \ref{exf2} so that the set
	$
	Z_0:=\{b\in M\mid f_\chi(b)=0\}
	$
	is a connected component of $M$.
	
	Applying Lemma \ref{exf1} to the formal manifold $Z_0$, we have a formal function $f\in \CO(Z_0;\R)$
	such that $\{b\in Z_0\mid f(b)\le C\}$ is compact for all $C\in \R$. By extension by zero, we also view $f$ as an element of $\CO(M)$. Put
	\[
	f':=(f-\chi(f))\cdot (f-\overline{\chi(f)}). %\in\CO(M).
	\]
	Then $f'\in \CO(M;\R)$,  $\chi(f')=0$, $f'(b)\geq 0$ for all $b\in M$, and  the set
	\[
	\{a\in Z_0\mid f'(a)=0\}
	\]
	is compact.

	Assume by contradiction that $\chi\neq  \mathrm{Ev}_a$ for all $a\in M$. Then for every $a\in M$, there is an element $f_{\chi}^a\in \CO(M)$ that satisfies the three conditions in Lemma \ref{threechi}. Since $f_{\chi}^a(a)=1$ for all $a\in M$, there  exist finitely many elements $a_1, a_2, \dots, a_s\in Z_0$ such that
	\[
	\bigcup_{i=1}^s \{b\in Z_0\mid f_{\chi}^{a_i}(b)>0\} \supset   \{a\in Z_0\mid f'(a)=0\}.
	\]
	
	Now we define
	\[
	f'':=f_\chi+f'+\sum_{i=1}^s f_{\chi}^{a_i}.
	\]
	Then $\chi(f'')=0$ and $f''(b)>0$ for all $b\in M$.
	Lemma \ref{lem:inverse} implies that  $f''$ is invertible in $\CO(M)$, which contradicts  $\chi(f'')=0$. \qed
	\vspace{3mm}
	
	% delete "to the fact that ref  Jacobson knapp and vogan"

	%\begin{lemd}\label{formk}
	%Let $N$ be a smooth manifold. Define a sheaf $\CO_{N,k}$ on $N$ of $\BC$-algebras by setting
	%\[
	%  \CO_{N,k}(U):=\RC^\infty(U)[[y_1, y_2, \cdots, y_k]]\qquad (\textrm{the formal power series ring})
	%\]
	%for  every open subset $U$ of $N$ (with the obvious restriction maps).  Then $(N, \CO_{N,k})$ is a $\BC$-locally ringed space with trivial residue fields.
	
	%\end{lemd}
	%\begin{proof}
	%Let $x\in N$. Then we have an algebra homomorphisms
	%\[
	% (\CO_{N,k})_x\xrightarrow{\textrm{the constant term}}\varinjlim_{U\textrm{is an open neighborhood of }x} \RC^\infty(U)\xrightarrow{\textrm{evalutiona at $x$}}  \BC.
	% \]
	% Write $(\m_{N,k})_x$ for the kernel of this homomorphism. It is clear that every element in  $(\CO_{N,k})_x\setminus (\m_{N,k})_x$ is invertible. This implies the lemma.
	
	%\end{proof}

	\subsection{Continuous characters of $\CO(M)$}
	
	For every connected component $Z$ of $M$, we view the identity element $1_Z\in \CO(Z)$ as an element of $\CO(M)$ via extension by zero. Then $1_Z\in \CO(M)$ is an indecomposable idempotent in $\CO(M)$. Namely, it is an idempotent in the sense that $1_Z^2=1_Z$, and it is indecomposable in the sense that it is nonzero, and is not the sum of two nonzero idempotents.  Moreover,
	it is obvious that every indecomposable idempotent in $\CO(M)$ is of the form $1_Z$ for a unique connected component $Z$ of $M$.
	
	By Theorem \ref{characterm2}, we know that if the cardinality of $\pi_0(M)$ does not exceed the cardinality of $\R$, then every character of $\CO(M)$ is continuous. In general, we have the following characterization of the continuity of a character of $\CO(M)$.
	
	\begin{prpd}\label{contccm}
		Let $\chi: \CO(M)\rightarrow \BC$ be a character. Then the following three conditions are equivalent to each other.
		\begin{itemize}
			\item[(a)]
			The character $\chi$ is continuous.
			\item[(b)]
			The character $\chi$  vanishes on $(1-e)\cdot \CO(M)$ for some indecomposable idempotent $e$ in $\CO(M)$.
			\item[(c)]
			The equality $\chi=\mathrm{Ev}_a$ holds for some $a\in M$.
		\end{itemize}
		When these conditions are satisfied, the elements $e\in \CO(M)$ and $a\in M$ as above are unique.
		
	\end{prpd}
	
	\begin{proof}
		We first prove that (a) implies (c).  Suppose that $\chi$ is continuous.
		Since
		\[
		\CO(M)=\prod_{Z\in \pi_0(M)} \CO(Z)
		\]
		as LCS,
		the continuity of $\chi$ implies that  there is an open  closed subset $Z'$ of $M$, consisting of
		finitely many connected components of $M$, such that $\chi$ factors through the restriction map
		\[
		\CO(M)\rightarrow \CO(Z').
		\]
		Write $\chi': \CO(Z')\rightarrow \BC$ for the induced character. By Theorem \ref{characterm2}, $\chi'$ equals the evaluation map at some $a\in Z'$. Thus $\chi=\mathrm{Ev}_a$ and we have proved that (a) implies (c).
		
		By using Theorem \ref{characterm2}, it is easy to see that (b) implies (c). The  uniqueness assertion as well as the  implications (c)$\Rightarrow $(a) and (c)$\Rightarrow $(b) are obvious.
	\end{proof}
	
	\subsection{The formal spectra}
	In this subsection, we reconstruct the formal manifold $(M,\CO)$ from the topological $\C$-algebra $\CO(M)$.
	We start with the following definition.
	
	\begin{dfn}
		A formal $\BC$-algebra is a topological $\BC$-algebra that is topologically isomorphic to $\CO(M)$ for some formal manifold $(M, \CO)$.
	\end{dfn}
	All topological $\BC$-algebras together with all continuous homomorphisms between them form a category. 
	All formal $\BC$-algebras form a full subcategory of the category of topological $\BC$-algebras, and thus they also form a category.
	
	Let $A$ be a formal $\BC$-algebra.
	Define the formal spectrum of $A$ to be the set
	\be\label{eq:specf}
	\mathrm{Specf}(A):=\{\textrm{continuous character of $A$}\}.
	\ee
	Equip $\mathrm{Specf}(A)$ with the coarsest topology such that the map
	\[
	\mathrm{Specf}(A)\rightarrow \BC, \quad \chi\mapsto \chi(f)
	\]
	is continuous for all $f\in A$.   For every $\chi\in \mathrm{Specf}(A)$, write $A_\chi$ for the localization of the ring $A$ at the maximal ideal $\ker(\chi)$.
	
	Let $U$ be an open subset of $\mathrm{Specf}(A)$. We  say that an element
	\[f=\{f_\chi\}_{\chi\in U} \in \prod_{\chi\in U} A_\chi\]
	is locally represented by $A$ if for every $\chi_0\in U$,
	there is an open neighborhood $U_0$ of $\chi_0$ in $U$,
	and an element $f_0\in A$ such that  $\{f_\chi\}_{\chi\in U_0}$ equals the image of $f_0$ under the canonical map
	\[
	A\rightarrow \prod_{\chi\in U_0}A_\chi.
	\]
	Define
	\be \label{eq:OA}
	\CO_A(U):=\left\{f\in \prod_{\chi\in U} A_\chi\mid f\textrm{ is locally represented by $A$ }\right\}.
	\ee 
	Together with the restriction maps, $\CO_A$ is a sheaf of $\BC$-algebras over $\mathrm{Specf}(A)$.
	
	The rest of this subsection is devoted to a proof of the following theorem.
	
	\begin{thmd}\label{ftom}
		The pair $(\mathrm{Specf}(A), \CO_A)$ is a formal manifold.
		Furthermore, if $A=\CO(M)$, then there is an identification $(\mathrm{Specf}(A), \CO_A)=(M, \CO)$ of formal manifolds.
	\end{thmd}
	
	We begin with the following lemma, which asserts that the topology on $M$ is determined by the space $\CO(M)$.
	\begin{lemd}\label{ftom2}
		The  map
		\be\label{topemb}
		\iota_M:\  M\rightarrow \prod_{f\in \CO(M)}  \BC, \quad a\mapsto \{f(a)\}_{f\in \CO(M)}
		\ee
		is a  topological embedding.
	\end{lemd}
	
	\begin{proof}
		The map \eqref{topemb} is clearly continuous and injective.
		In what follows we prove that the inverse  map
		\[
		\iota_M^{-1}:\ \iota_M(M)\rightarrow M
		\]
		is also continuous, which implies the lemma.
		
		Let $a\in M$ and let $U$ be an open neighborhood of $a$ in $M$. Take a formal function $f_0\in \CO(M)$ such that its support is contained in $U$ and $f_0(a)=1$.
		Then
		\[
		U':= \left\{\{c_f\}_{f\in \CO(M)}\in  \prod_{f\in \CO(M)}  \BC \mid (c_{f_0}-1)\cdot (\overline{c_{f_0}}-1)<1\right\}
		\]
		is an open neighborhood of $\iota_M(a)$ in $\prod_{f\in \CO(M)}  \BC$ such that $\iota_M^{-1}(U')\subset U$, as required.
		
	\end{proof}

	\begin{lemd}\label{ftom3}
		Let $a\in M$. Then as a $\BC$-algebra, the stalk $\CO_a$ is canonically isomorphic to the localization of $\CO(M)$ at the maximal ideal
		\[
		\m_a':= \{f\in \CO(M)\mid f(a)=0\}.
		\]
		
	\end{lemd}
	
	\begin{proof}
		Write
		\[
		\rho_a: \CO(M)\rightarrow \CO(M)_{\m_a'}
		\]
		for the localization map. For every $f\in \CO(M)$, we have that
		\begin{eqnarray*}
			&&f\in \ker(\rho_a)\\
			&\Longleftrightarrow& f_0 f=0\textrm{ for some $f_0\in \CO(M)\setminus \m_a'$}\\
			&\Longleftrightarrow& f|_U=0 \textrm{ for some open neighborhood $U$ of $a$ in $M$}.
		\end{eqnarray*}
		Moreover, for every $f_0\in \CO(M)\setminus \m_a'$, there is an element $f_1\in \CO(M)$ such that $\rho_a(f_0f_1-1)=0$. This implies that the map $\rho_a$ is surjective.
		
		The lemma then follows by noting  that the canonical map
		\[
		\CO(M)\rightarrow \CO_a
		\]
		is also surjective (see Corollary \ref{softo1}) and its kernel equals that of $\rho_a$.
		
	\end{proof}

	\noindent\textbf{Proof of Theorem \ref{ftom}:}
	Assume that $A=\CO(M)$. By Proposition \ref{contccm}, we have a bijection
	\[
	\mathrm{Ev}:   M\rightarrow \mathrm{Specf}(A), \quad a\mapsto \mathrm{Ev}_a.
	\]
	This is in fact a homeomorphism by Lemma \ref{ftom2}. In the notation of Lemma \ref{ftom3}, we have that $\CO_a=\CO(M)_{\m_a'}$ for every $a\in M$.
	Also, note that for every open subset $U$ of $M$,
	\[
	\CO(U)=\left\{f\in \prod_{a\in U} \CO_a\mid f\textrm{ is locally represented by $\CO(M)$ }\right\}.
	\]
	Thus with respect to the homeomorphism $\mathrm{Ev}$,  the sheaf $\CO$ on $M$ coincides with the sheaf $\CO_A$ on $\mathrm{Specf}(A)$. The proof of Theorem \ref{ftom} is now finished. \qed
	\vspace{3mm}

	\subsection{An equivalence of categories}
	In this subsection, we prove that the category of formal manifolds is equivalent to the opposite category of that of formal $\C$-algebras.
	
	Let $\psi: A'\rightarrow A$ be a continuous homomorphism of formal $\BC$-algebras.
	Then it induces a map
	\be\label{eq:SpectoSpec}
	\overline\psi:\, \mathrm{Specf}(A)\rightarrow \mathrm{Specf}(A'), \quad \chi\mapsto \chi\circ \psi.
	\ee
	It is easily seen that this map is continuous. For every $\chi\in \mathrm{Specf}(A)$, $\psi$ also induces a homomorphism
	\[
	\psi_{\chi}^*:\, A'_{\psi(\chi)}\rightarrow A_\chi.
	\]
	
	\begin{lemd}\label{pullbackff}
		Let $U'\subset \mathrm{Specf}(A')$ be an open subset and put $U:=\overline\psi^{-1}(U')\subset \mathrm{Specf}(A)$. Then  there exists a unique $\BC$-algebra homomorphism
		\be\label{psist}
		\psi^*_{U'}: \CO_{A'}(U')\rightarrow \CO_A(U)
		\ee
		such that the diagram
		\[
		\begin{CD}
			\CO_{A'}(U') @> \psi^*_{U'}  >>  \CO_A(U)\\
			@V  VV           @V V V\\
			A'_{\psi(\chi)} @> \psi_{\chi}^*>>  A_{\chi} \\
		\end{CD}
		\]
		commutes for all $\chi\in U$.
		
	\end{lemd}
	\begin{proof}
		The uniqueness follows from the fact that the canonical map $\CO_A(U)\rightarrow \prod_{\chi\in U} A_\chi$ is injective. Note that if an element
		\[
		f'=\{f'_{\chi'}\}_{\chi'\in U'}\in \prod_{\chi'\in U'} A'_{\chi'}
		\]
		is locally represented by $A'$, then
		\[
		f:=\{\psi_{\chi}^*(f'_{\overline\psi(\chi)})\}_{\chi\in U}\in \prod_{\chi\in U} A_{\chi}
		\]
		is locally represented by $A$. Furthermore, the map 
		\[
		\psi^*_{U'}: \CO_{A'}(U')\rightarrow \CO_A(U),\quad f'\mapsto f
		\]
		is a $\C$-algebra homomorphism. Thus the existence assertion of the lemma also holds.
		
	\end{proof}

	The family $$\{\psi^*_{U'}: \CO_{A'}(U')\rightarrow \CO_A(U)\}_{\, U'\textrm{ is an open subset of }M'}$$ as in \eqref{psist} is compatible with respect to the restrictions. Thus we get a sheaf homomorphism
	\[
	\psi^*: \overline \psi^{-1} \CO_{A'}\rightarrow \CO_A.
	\]
	This sheaf homomorphism induces local homomorphisms of the stalks. Then we obtain in this way
	a morphism
	\be\label{eq:specphi} \mathrm{Specf}(\psi)=(\overline \psi, \psi^*):\ (\mathrm{Specf}(A),\CO_A)\rightarrow (\mathrm{Specf}(A'), \CO_{A'})\ee  of formal manifolds. The assignment
	$\psi\mapsto \mathrm{Specf}(\psi)$ is compatible with the compositions, and if $\psi$ is an identity homomorphism,
	then $\mathrm{Specf}(\psi)$ is also an identity morphism. In conclusion, we get a functor
	\be\label{functor1}
	A\mapsto (\mathrm{Specf}(A),\CO_A)
	\ee
	from the opposite category of the category of formal $\BC$-algebras to the category of formal manifolds.
	
	On the other hand, by Theorem \ref{thm:autocon}, we have a functor
	\be\label{functor2}
	(M, \CO)\mapsto \CO(M)
	\ee
	from  the category of formal manifolds to the opposite category of the category of formal $\BC$-algebras.
	Then we have  the following result, which implies Theorem \ref{thmmain1}.
	
	\begin{thmd}\label{thmqe}
		The functors \eqref{functor1} and \eqref{functor2}
		are quasi-inverse of each other.
	\end{thmd}
	\begin{proof}
		
		This is implied by
		Theorem \ref{ftom}.
		
	\end{proof}

	\begin{remarkd} 
		Assume that 
		\be \label{eq:cardpi0M} 
		\text{the cardinality of $\pi_0(M)$ does not exceed the cardinality of $\R$}.
		\ee
		By Theorem \ref{characterm2}, it is  not difficult to show that the formal spectrum $\mathrm{Specf(\CO(M))}$
		coincides with the usual spectrum 
		\[\mathrm{Spec}(\CO(M)):=\{\text{character of $\CO(M)$}\}\] of the $\C$-algebra $\CO(M)$  (see \cite[Section 8.1]{N}) as topological spaces. 
		Similar to the case of smooth manifolds (see \cite[Section 7]{N} for example),  in such case, $(M,\CO)$ can be reconstructed from the spectrum $\mathrm{Spec}(\CO(M))$ by Theorem \ref{ftom}.
		Furthermore, for any formal manifold $(M',\CO')$ (whether or not it satisfies the condition \eqref{eq:cardpi0M}),
		it is easy to verify that every $\C$-algebra homomorphism from  $O(M)$ to
		$O'(M')$  is automatically continuous (see  \cite[Page 88, Proposition (a)]{N}).
	\end{remarkd}
	%Assume that \be \label{eq:cardpi0M} \text{the cardinality of $\pi_0(M)$ does not exceed the cardinality of $\R$}.\eeBy Theorem \ref{characterm2}, the formal spectrum $\mathrm{Specf(\CO(M))}$coincides with the usual spectrum \[\mathrm{Spec}(\CO(M)):=\{\text{characters of $\CO(M)$}\}\] of the $\C$-algebra $\CO(M)$  (see \cite[Section 8.1]{N}). Just as the smooth manifold case (see \cite[Section 7]{N} for example),  in this case $(M,\CO)$ can be reconstructed from  the spectra $\mathrm{Spec}(\CO(M))$ by Theorem \ref{ftom}.Furthermore, it is easy to see  that every $\C$-algebra homomorphism from  $O(M)$ to $O'(M')$  is automatically continuous, where $(M',\CO')$ is another formal manifold (may or may not satisfy the assumption \eqref{eq:cardpi0M}).

	\section{Vector-valued formal functions and products of formal manifolds} \label{sec:prod}
	In this section, we introduce the notion of vector-valued formal functions.
	Using this, we prove the existence of finite products in the category of formal manifolds.
	
	\subsection{Vector-valued smooth functions}\label{appendixB1} In this subsection, we review some basics of topological tensor products and vector-valued smooth functions.
	Throughout this subsection, let  $E$ be an LCS which may or may not be Hausdorff.
	\begin{dfn}\label{de:strictdense}  Let $A$ be a subset of $E$.
		
		\noindent (a) The subset $A$ is called quasi-closed if it contains all limit points of the bounded subsets in it.
		
		\noindent (b) The quasi-closure of $A$ is defined as the intersection of all the quasi-closed subsets of $E$  that contain $A$.
		
		\noindent (c) The subset $A$ is called strictly dense in $E$ if its quasi-closure is $E$.
	\end{dfn}
	
	The intersection of arbitrarily many quasi-closed subsets is also quasi-closed. In particular, the quasi-closure of every subset in $E$ is quasi-closed. 
	
	\begin{dfn} An LCS is said to be complete (resp.\,quasi-complete) if it is Hausdorff, and every
		Cauchy net (resp.\,bounded Cauchy net) in it has a limit point. 
	\end{dfn}
	\begin{dfn} 
		A completion of $E$ is a complete LCS $\wh E$ together with a continuous linear map $\hat \iota: E\rightarrow \wh E$ with the following property: for every complete LCS $F$ and every continuous linear map $f: E\rightarrow F$, there is a unique continuous linear map $f':\wh E\rightarrow F$ such that the  diagram
		\[\xymatrix{ &\wh E\ar[d]^{f'}\\ E\ar[ur]^{\hat \iota}\ar[r]^{f}&F}\]
		commutes.
	\end{dfn}
	
	It is known that the LCS $E$ has a unique completion, and the above map $\hat \iota$ is a topological embedding with dense image when  $E$ is Hausdorff. See \cite[Theorem 5.2]{Tr} for example.  When $E$ is not Hausdorff, the completion 
	$\widehat{E}$ is precisely the completion of $E/\overline{\{0\}}$.  Here $\overline{\{0\}}$ is the closure of the zero space in $E$. 
	%Similarly, in this paper, when we refer to a complete  LCS, we mean it is Hausdorff, and every Cauchy net within it possesses a limit point.
	\begin{dfn} 
		A quasi-completion of $E$ is a quasi-complete  LCS $\wt E$ together with a continuous linear map $\tilde\iota: E\rightarrow \wt E$ with the following property: for every quasi-complete  LCS $F$ and every continuous linear map $f: E\rightarrow F$, there is a unique continuous linear map $f':\wt E\rightarrow F$ such that the  diagram
		\[\xymatrix{ &\wt E\ar[d]^{f'}\\ E\ar[ur]^{\tilde{\iota}}\ar[r]^{f}&F}\]
		commutes.
	\end{dfn}
	
	The LCS $E$ has a unique quasi-completion, and when $E$ is Hausdorff the quasi-completion $\wt E$ agrees with the quasi-closure of $\hat \iota(E)$ in $\wh{E}$. 
	%When $E$ is Hausdorff, we obviously identify $E$ with a subspace of its completion $\wh{E}$. Then the quasi-completion of $E$ is identified with the quasi-closure of $E$ in $\wh{E}$. 
	See \cite[Page 91]{Sc} for details.
	When $E$ is not Hausdorff, the quasi-completion $\widetilde{E}$ agrees with the quasi-completion  of $E/\overline{\{0\}}$.

	The following lemma is straightforward.
	
	\begin{lemd}\label{lem:strictlydense}
		Let  $E_0$ be a complete (resp.\,quasi-complete) LCS, and let $f:\, E\rightarrow E_0$ be a linear topological embedding such that  $f(E)$ is  dense (resp.\,strictly  dense) in $F$. Then $E_0$ is the completion (resp.\,quasi-completion) of $E$.
	\end{lemd}
	
	We denote by $E\otimes_{\pi} F$  the projective tensor product of  $E$ and $F$, where $F$ is another LCS.
	As mentioned in the Introduction, the quasi-completion and completion of $E\otimes_{\pi} F$ are denoted as $E\widetilde\otimes_{\pi} F$ and $E\widehat\otimes_{\pi} F$, respectively.

	The following result is about projective tensor products of products of LCS.
	
	\begin{lemd}\label{lem:protensor} Let $E=\prod_{i\in I} E_i $ and $F=\prod_{j\in J}F_j$ be two products of Hausdorff LCS.
		Then the canonical linear map
		\be\label{eq:protensor} E\otimes_\pi F\rightarrow
		\prod_{(i,j)\in I\times J}E_i\otimes_\pi F_j\ee is a topological embedding, and
		induces the following LCS identifications:
		\be\label{eq:quasiprotensor} E\widetilde\otimes_\pi F
		=\prod_{(i,j)\in I\times J}E_i\widetilde\otimes_\pi F_j\quad\text{and}\quad  E\widehat\otimes_\pi F
		=\prod_{(i,j)\in I\times J}E_i\widehat\otimes_\pi F_j.\ee
	\end{lemd}
	\begin{proof}
		The first assertion of the lemma and the second identification in \eqref{eq:quasiprotensor} are well-known (see \cite[\S 15.4, Theorem 1]{Ja}).
		Using this identification, the linear map  \eqref{eq:protensor} restricts to a topological linear  isomorphism
		\[
		\left(\bigoplus_{i\in I}E_i\right)\otimes_{\pi} \left(\bigoplus_{j\in I}F_j\right)\rightarrow \bigoplus_{(i,j)\in I\times J} (E_i\otimes_\pi F_j).
		\]
		Here the direct sums are equipped with the subspace topologies of the direct products. Meanwhile, since a product of quasi-complete %Hausdorff
		LCS is still  quasi-complete (see \cite[\S 3.2, Proposition 6]{Ja}),
		there is a linear topological embedding
		\bee
		E\widetilde\otimes_\pi F \rightarrow \prod_{(i,j)\in I\times J}(E_i\widetilde\otimes_\pi F_j)
		\eee
		induced by \eqref{eq:protensor}. As $\bigoplus_{(i,j)\in I\times J} (E_i\otimes_\pi F_j)$
		is strictly dense in $\prod_{(i,j)\in I\times J}(E_i\widetilde\otimes_\pi F_j)$,
		the first identification in \eqref{eq:quasiprotensor} then follows.
	\end{proof}

	We also denote by $E\otimes_\varepsilon F$ the epsilon (injective) tensor product of $E$ and $F$ (see \cite{Gr}). 
	Similarly, the quasi-completion and completion of $E\otimes_{\varepsilon} F$ are denoted as $E\widetilde\otimes_{\varepsilon} F$ and $E\widehat\otimes_{\varepsilon} F$, respectively. 
	When $E$ or $F$ is nuclear, it is a classical result of Grothendieck that their projective tensor product coincides with the epsilon tensor product (see \cite[Theorem 50.1]{Tr}).
	In this case, we
	will simply write
	\bee
	E\widetilde\otimes F:=E\widetilde\otimes_{\pi} F=E\widetilde\otimes_{\varepsilon} F
	\quad
	\text{and}
	\quad
	E\widehat\otimes F:=E\widehat\otimes_{\pi} F=E\widehat\otimes_{\varepsilon} F.
	\eee
	
	In the rest part of this subsection, assume  that $E$ is a quasi-complete LCS, and let 
	$N$ be a smooth manifold.
	Write $\RC^\infty(N;E)$ for the space of all $E$-valued  smooth functions on $N$.
	Endow $\RC^\infty(N;E)$ with
	the topology defined by the seminorms 
	\[
	\abs{f}_{D, \abs{\,\cdot\,}_\mu}:= \sup_{a\in N}\abs{(Df)(a)}_\mu\qquad (f\in \RC^\infty(N;E)),
	\]
	where $D$ is a compactly supported differential operator on $N$, and $\abs{\,\cdot \,}_\mu$ is a continuous seminorm on $E$.
	Then $\RC^\infty(N;E)$ is a quasi-complete LCS.

	Identify $\RC^\infty(N)\otimes E$ as a linear subspace of $\RC^\infty(N;E)$, which consists of the smooth functions
	whose image is contained in a finite-dimensional subspace of $E$.
	The following result is proved in 
	\cite[Th\'{e}mr\`{e}me 1]{Sc}. 
	\begin{prpd} 
		We have that
		\be\label{eq:vvsfiso}
		\RC^\infty(N)\widetilde \otimes E=\RC^\infty(N;E)
		\ee as LCS, 
		and if $E$ is complete, then
		\be
		\label{eq:vvsfiso2}
		\RC^\infty(N)\wt\otimes E=\RC^\infty(N;E)=\RC^\infty(N)\wh\otimes E
		\ee
		as LCS.   
	\end{prpd}

	\begin{exampled}\label{ex:E-valuedpowerseries}Let $k\in \BN$ and write
		\[
		E[[y_1, y_2, \dots, y_k]]:=\left\{\sum_{J\in \BN^k} v_J y^J\,:\, v_J\in E \textrm{ for all }J\in \BN^k\right\}
		\]
		for the space of the formal power series.
		Equip  $E[[y_1, y_2, \dots, y_k]]$ with the term-wise convergence topology, which is defined  by the seminorms
		\be \label{eq:defabsnuJ}\abs{\,\cdot\,}_{\nu,J}:\quad \sum_{I\in \BN^k} f_I y^I \mapsto \abs{f_J}_{\nu},\ee
		where $\abs{\,\cdot\,}_\nu$ is a continuous seminorm on $E$ and $J\in \BN^k$.
		When  $N=\BN^k$, we have  LCS identifications
		\begin{eqnarray*}
			E[[y_1, y_2, \dots, y_k]]
			=\RC^\infty(N; E)
			=\RC^\infty(N)\widetilde \otimes E
			=E \widetilde \otimes \BC[[y_1, y_2, \dots, y_k]],
		\end{eqnarray*}
		and \[ E[[y_1, y_2, \dots, y_k]]=E \widehat \otimes \BC[[y_1, y_2, \dots, y_k]]\] is complete provided that $E$ is complete.
		% Furthermore, if $E$ is a Fr\'echet space, then
		% \[E[[y_1, y_2, \dots, y_k]]
		% =E \widetilde \otimes_{\mathrm{i}} \BC[[y_1, y_2, \dots, y_k]]= E \widehat \otimes_{\mathrm{i}} \BC[[y_1, y_2, \dots, y_k]].\]
	\end{exampled}

	\begin{exampled} Let $L$ be another smooth manifold. Then  we have the following LCS identifications (see \cite[Proposition 14]{Sc}): 
		\begin{eqnarray}\label{eq:schwartzkernel1}
			&&\RC^\infty(N)\widetilde\otimes\RC^\infty(L)=\RC^\infty(N;\RC^\infty(L))=\RC^\infty(N)\widehat\otimes \RC^\infty(L)
			=\RC^\infty(N\times L).
		\end{eqnarray}
	\end{exampled}

	\subsection{Vector-valued formal functions}
	Throughout this subsection, let $E$ be a quasi-complete  LCS. 
	For every open subset $U$ of $M$,  write %{\footnotesize define}
	\[
	\CO_E(U):=\CO(U)\widetilde \otimes E.
	\]

	Motivated by Schwartz's characterization of $E$-valued smooth functions (see \eqref{eq:vvsfiso}), we make the following definition. 
	
	\begin{dfn}
		An element in $\CO_E(M)$ is called an $E$-valued formal function on $M$.
	\end{dfn}
	
	With the obvious restriction maps, the assignment 
	\[
	\CO_E:\ U\mapsto \CO_E(U)\quad (\text{$U$ is an open subset of $M$})
	\]
	forms a presheaf of complex vector spaces on $M$.
	The main goal of this subsection is to prove the following theorem.
	\begin{thmd}\label{thmvvf}
		The presheaf $\CO_E$ is a sheaf. Moreover, if $E$ is complete, then for every open subset $U$ of $M$,
		\[
		\CO_E(U)=\CO(U)\widehat \otimes E
		\]
		as LCS.
	\end{thmd}

	We start with the following special case.

	\begin{lemd}\label{evalued}
		Theorem \ref{thmvvf} holds when $M=N^{(k)}$, where   $N$ is  a smooth manifold and $k\in \BN$.
	\end{lemd}
	
	\begin{proof}
		Let $U$ be an open subset of $M$. By applying \eqref{eq:vvsfiso} and Example \ref{ex:E-valuedpowerseries}, we obtain that
		\begin{eqnarray*}
			\CO_E(U)&=&\CO(U)\widetilde \otimes E\\
			&=& \RC^\infty(U)\widetilde \otimes \BC[[y_1, y_2, \dots, y_k]]\widetilde \otimes E\\
			&=&\RC^\infty(U;E[[y_1, y_2, \dots, y_k]]).
		\end{eqnarray*}
		Thus $\CO_E$ is a sheaf. Moreover, if $E$ is complete, then $\RC^\infty(U;E[[y_1, y_2, \dots, y_k]])$
		is complete by \eqref{eq:vvsfiso2}, and so $\CO_E(U)=\CO(U) \widehat \otimes E$.
	\end{proof}

	Write $\widetilde{\CO}_E$ for the sheaf associated to the presheaf $\CO_E$. Lemma \ref{evalued} implies that $\widetilde{\CO}_E(U)=\CO_E(U)$ whenever the open set $U$ is a
	chart. In general, the sheaf property implies that
	\[
	\widetilde{\CO}_E(U)=\varprojlim_{V} \CO_E(V),
	\]
	where $V$ runs over all charts of $U$. Using the above equality, we view $\widetilde{\CO}_E(U)$ as an LCS with the projective limit topology. It is quasi-complete by \cite[\S\,19.10 (2)]{Ko1}.
	If the open set $U$ is a chart, then
	\be\label{equalvt}
	\widetilde{\CO}_E(U)=\CO_E(U)
	\ee
	as LCS.
	
	The following lemma generalizes Lemma \ref{lemr0}.
	
	\begin{lemd}\label{lemrv0}
		Let $U$ be an open subset of $M$. Then the restriction map
		\be\label{redo0v}
		\widetilde{\CO}_E(M)\rightarrow \widetilde{\CO}_E(U)
		\ee
		is continuous.
		
	\end{lemd}
	
	\begin{proof}
		It suffices to show that for every  chart $V$ of $U$, the canonical map
		\[
		\widetilde{\CO}_E(M)\rightarrow \CO_E(V)
		\]
		is continuous. But this is self-evidence since $V$ is also a chart of $M$.
	\end{proof}
	
	Similar to Lemma \ref{lemr1}, we have the following result.
	
	\begin{lemd}\label{lemrv1}
		Let $\{U_\gamma\}_{\gamma\in \Gamma}$ be an open cover of $M$. Then the linear map
		\be\label{redov}
		\widetilde{\CO}_E(M)\rightarrow \prod_{\gamma\in \Gamma} \widetilde{\CO}_E(U_\gamma), \quad f\mapsto \{f|_{U_\gamma}\}_{\gamma\in \Gamma}
		\ee
		is a closed topological embedding.
		
	\end{lemd}
	\begin{proof}
		The sheaf property of $\widetilde{\CO}_E$ and Lemma \ref{lemrv0} imply that the map \eqref{redov} is continuous, injective and its image is closed in the codomain.
		
		Let $\{f_i\}_{i\in I}$ be a net in  $ \widetilde{\CO}_E(M)$ whose image under the map \eqref{redov}  converges to zero. Then 
		\be\label{limfiu}
		\lim_{i\in I} (f_i|_{U_\gamma})=0\quad \textrm{for all }\gamma\in \Gamma.
		\ee
		To finish the proof, we only need to show that
		\be\label{limiifi}
		\lim_{i\in I} (f_i|_U)=0
		\ee
		for all charts $U$ of $M$.
		
		By Lemma \ref{lemr1}, the linear map
		\[
		\CO(U)\rightarrow \prod_{\gamma\in \Gamma} \CO(U\cap U_\gamma)
		\]
		is a topological embedding. Since the epsilon tensor product of two linear topological embeddings of  Hausdorff LCS is a linear topological embedding (see the proof of \cite[Proposition 43.7]{Tr}), Lemma \ref{lem:protensor} and \eqref{equalvt}   imply that 
		the linear map
		\be\label{temb}
		\widetilde{\CO}_E(U) \rightarrow \prod_{\gamma\in \Gamma} \widetilde{\CO}_E(U\cap U_\gamma)\ \ \textrm{ is also a topological embedding.}
		\ee
		Now, Lemma \ref{lemrv0} and \eqref{limfiu} imply that
		\be\label{limfiu2}
		\lim_{i\in I} (f_i|_{U\cap U_\gamma})=0\quad \textrm{for all }\gamma\in \Gamma.
		\ee
		Hence  \eqref{limiifi} holds by \eqref{temb}.
		
	\end{proof}
	
	% Using the above lemma, we have the following result about convergence and boundedness.
	% \begin{lemd}\label{lemrv122}
		% (a) A net $\{f_i\}_{i\in I}$ in $\widetilde{\CO}_E(M)$ converges to an element $f\in \widetilde{\CO}_E(M)$ if and only if
		% for every relatively compact open subset $U$ of $M$,
		% the net $\{(f_i)|_U\}_{i\in I}$ converges to $f|_U$ in $\widetilde{\CO}_E(U)$.
		
		% \noindent (b) A subset $B$ of  $\widetilde{\CO}_E(M)$ is bounded if and only if for every relatively compact open subset $U$ of $M$, the subset
		% \[
		%   B|_U:=\{g|_U\mid g\in B\}
		% \]
		% of $\widetilde{\CO}_E(U)$ is bounded.
		% \end{lemd}
	% \begin{proof}
		% By Lemma \ref{lemrv1}, the canonical linear  map
		% \be\label{redov1}
		%    \widetilde{\CO}_E(M)\rightarrow \prod_{U\textrm{ is a  relatively compact open subset of $M$}} \widetilde{\CO}_E(U)
		% \ee
		% is a topological embedding. This implies the lemma.
		
		% \end{proof}

	Note that $\widetilde{\CO}_E$ is naturally a sheaf of $\CO$-modules, and the module structure map
	\[
	\CO(M)\times \widetilde{\CO}_E(M)\rightarrow \widetilde{\CO}_E(M), \quad (g,f)\mapsto gf
	\]
	is continuous and bilinear.
	\begin{lemd}\label{subtensor22}
		Let  $U$ be an open subset of $M$, and let $g\in \CO(M)$ be a formal function whose support is contained in $U$. Then the linear map  (see \eqref{eq:deffu})
		\be \label{eq:defmg}
		m_g:  \widetilde{\CO}_E(U)\rightarrow \widetilde{\CO}_E(M), \quad f\mapsto gf
		\ee
		is continuous.
		
	\end{lemd}
	\begin{proof}
		The lemma follows by applying Lemma \ref{lemrv1} to the open cover $\{U, M\setminus \mathrm{supp}\,g\}$ of $M$.
		
	\end{proof}

	\begin{lemd}\label{subtensor}
		For every open subset $U$ of $M$, the canonical linear map
		\[
		\CO(U)\otimes_\pi E\rightarrow \widetilde{\CO}_E(U)
		\]
		is a topological embedding.
		
	\end{lemd}
	\begin{proof}
		Consider the commutative diagram
		\[
		\begin{CD}
			\CO(U)\otimes_\pi E@>  >>  \widetilde{\CO}_E(U)=\varprojlim_{V} \CO(V)\widetilde \otimes E \\
			@V  VV           @V V V\\
			\left(\prod_{V} \CO(V)\right)\widetilde\otimes E @ >> > \prod_{V} \left(\CO(V)\widetilde \otimes E\right),\\
		\end{CD}
		\]
		where $V$ runs over all charts of $M$ that are contained in $U$.
		By Lemma \ref{lemr1}, Proposition \ref{prop:OMnuclear} and the fact that the epsilon tensor product of two linear topological embeddings of Hausdorff LCS is a linear topological embedding (see the proof of \cite[Proposition 43.7]{Tr}), we see that the linear  map
		\[
		\CO(U)\otimes_\pi E=
		\CO(U)\otimes_\varepsilon E\rightarrow 
		\left(\prod_{V} \CO(V)\right)\otimes_\pi E
		=\left(\prod_{V} \CO(V)\right)\otimes_\varepsilon E
		\] 
		is a topological embedding. 
		This implies that
		the left vertical arrow
		is a linear topological embedding. By Lemma \ref{lem:protensor}, the bottom horizontal arrow is a topological linear isomorphism.
		Since
		the right vertical arrow is a linear topological embedding by Lemma \ref{lemrv1}, the top horizontal arrow is also
		a linear topological embedding. 
	\end{proof}
	
	By     %In view of
	Lemma \ref{subtensor}, we identify the LCS $\CO(U)\otimes_\pi E$ with a subspace of $\widetilde{\CO}_E(U)$.

	\begin{prpd}\label{clb}
		The subspace $\CO(M)\otimes_\pi E$ is strictly dense (see Definition \ref{de:strictdense}) in $\widetilde{\CO}_E(M)$. 
		
	\end{prpd}
	
	\begin{proof}
		Take an atlas  $\{U_\gamma\}_{\gamma\in \Gamma}$  of $M$, and  let  $\{g_\gamma\}_{\gamma\in \Gamma}$ be a partition of unity on $M$ subordinate to it.
		Let $A$ be a quasi-closed subset (see Definition \ref{de:strictdense}) of $\widetilde{\CO}_E(M)$ containing $\CO(M)\otimes E$.
		Since the inverse image of a quasi-closed set under a continuous linear map is still quasi-closed  (see \cite[Page 92, $2^\circ$)]{Sc}),
		it follows from Lemma \ref{subtensor22}  that  $m_{g_\gamma}^{-1}(A)$ (see \eqref{eq:defmg})
		is quasi-closed.
		This, together with \eqref{equalvt}, shows that $m_{g_\gamma}^{-1}(A)=\widetilde{\CO}_E(U_\gamma)$ for all $\gamma\in \Gamma$.
		In particular, for every $f\in \widetilde{\CO}_E(M)$, we have that $f_\gamma:= g_\gamma\cdot f|_{U_\gamma}\in A$.
		
		By Lemma \ref{lemrv1}, the canonical linear  map
		\[
		\widetilde{\CO}_E(M)\rightarrow \prod_{U\textrm{ is a  relatively compact open subset of $M$}} \widetilde{\CO}_E(U)
		\]
		is a topological embedding.
		This implies that
		\[B:=\left\{\sum_{\gamma\in \Gamma_0}f_\gamma\mid \Gamma_0\ \text{is a finite subset of}\ \Gamma\right\}\]
		is a bounded subset in $A$. Thus $A$ contains the closure of $B$. Then 
		\[
		f=\sum_{\gamma\in \Gamma}g_\gamma\, f=\sum_{\gamma\in \Gamma} f_\gamma
		\]
		belongs to $A$ since it belongs to the closure of $B$. 
		This completes the proof.
		
	\end{proof}

	\noindent\textbf{Proof of Theorem \ref{thmvvf}:}
	Let $U$ be an open subset of $M$.
	Applying Proposition \ref{clb} to the formal manifold $(U, \CO|_U)$, we know that
	$\CO(U)\otimes_\pi E$ is strictly dense in $\widetilde{\CO}_E(U)$. Therefore, by Lemma \ref{lem:strictlydense}, we have that  
	\[
	\widetilde{\CO}_E(U)=\CO(U)\widetilde \otimes E=\CO_E(U)
	\]  as LCS.
	This proves that $\CO_E$ is a sheaf.
	When $E$ is complete, Lemmas \ref{evalued} and \ref{lemrv1} imply that $\CO_E(U)$ is complete.
	This finishes the proof of Theorem \ref{thmvvf}.
	\qed
	\vspace{3mm}

	\subsection{Products of formal manifolds}
	
	In this subsection, let $(M_1, \CO_1)$ and $(M_2, \CO_2)$ be two formal manifolds.
	For open subsets $V_1\subset U_1$ of $M_1$, and $V_2\subset U_2$ of $M_2$,  we have the restriction map
	\[
	\CO_1(U_1)\widetilde \otimes \CO_2(U_2)\rightarrow
	\CO_1(V_1)\widetilde \otimes \CO_2(V_2).
	\]
	Using these restriction maps, for each open subset $U_3$ of 
	\[M_3:=M_1\times M_2 \quad \text{(as topological spaces),}\] we define
	\[
	\CO_3(U_3):=\varprojlim_{(U_1,U_2)} \CO_1(U_1)\widetilde \otimes \CO_2(U_2),
	\]
	where $(U_1,U_2)$ runs over all pairs of open subsets $U_1$ of $M_1$ and $U_2$ of $M_2$ such that $U_1\times U_2\subset U_3$. Then $\CO_3$ is obviously a presheaf of $\C$-algebras on $M_3$.
	
	This subsection is devoted  to a proof of the following theorem, which together with the fact
	that (see Theorem  \ref{thmvvf})
	\[\CO(M_3)=\CO_1(M_1)\widetilde \otimes \CO_2(M_2)=\CO_1(M_1)\widehat \otimes \CO_2(M_2)\]
	implies Theorem \ref{thmmain2}.
	
	\begin{thmd}\label{pfm00}
		The presheaf $\CO_3$ on $M_3$ is a sheaf, and $(M_3, \CO_3)$ is a formal manifold which is a product of $(M_1, \CO_1)$ and $(M_2, \CO_2)$ in the category of  formal manifolds. Moreover,
		\[
		(M_3,\underline{\CO_{3}})=(M_1, \underline{\CO_{1}}) \times (M_2, \underline{\CO_{2}})
		\]
		as smooth manifolds, and for $a_1\in M_1, a_2\in M_2$,
		\[
		\wh{\CO}_{3,a_3}=\wh{\CO}_{1,a_1}\wh{\otimes}\,\wh{\CO}_{2,a_2}=\wh{\CO}_{1,a_1}\widetilde{\otimes}\,\wh{\CO}_{2,a_2}\qquad (a_3:=(a_1,a_2))
		\]
		as topological $\C$-algebras.
	\end{thmd}
	
	We will prove the theorem through a sequence of lemmas.
	
	\begin{lemd}\label{prodvt}
		Assume that $M_1=N^{(k)}$ and $M_2=L^{(l)}$, where $N$ and $L$ are smooth manifolds and $k,l\in \BN$.
		Then $\CO_3$ is a sheaf,  $(M_3, \CO_3)$ is a formal manifold, and $M_3\cong (N\times L)^{(k+l)}$ as formal manifolds.
	\end{lemd}
	\begin{proof}
		Let $U_1$ and $U_2$ be open subsets of $M_1$ and $M_2$, respectively.
		It  follows from  Example \ref{ex:E-valuedpowerseries} and \eqref{eq:schwartzkernel1} that  
		\begin{eqnarray*}
			\CO_3(U_1\times U_2)&=&  \CO_1(U_1)\widetilde \otimes \CO_2(U_2)\\
			&=& \RC^\infty(U_1)[[y_1, y_2, \dots, y_k]]\widetilde \otimes \RC^\infty(U_2)[[y_1, y_2,\dots, y_l]]\\
			&=& \RC^\infty(U_1)\widetilde \otimes \BC[[y_1, y_2, \dots, y_k]]\widetilde \otimes \RC^\infty(U_2)\widetilde \otimes \BC [[ y_1, y_2,\dots, y_l]]\\
			&=& \RC^\infty(U_1\times U_2)\widetilde \otimes \BC[[y_1, y_2, \dots, y_k,  y_{k+1}, y_{k+2},\dots, y_{k+l}]]\\
			&=& \RC^\infty(U_1\times U_2)[[y_1, y_2, \dots, y_k,  y_{k+1}, y_{k+2},\dots, y_{k+l}]]
		\end{eqnarray*}
		as topological $\BC$-algebras. Hence by the sheaf property, we conclude that $\CO_3=\CO_{N\times L}^{(k+l)}$. This proves the lemma.
	\end{proof}
	
	Write $\widetilde{\CO}_3$ for the sheafication of $\CO_3$. We will show that  $\widetilde{\CO}_3=\CO_3$.

	\begin{lemd}\label{prodvt2}
		If $M_1=N^{(k)}$ with $N$ a smooth manifold and $k\in \BN$, then $\CO_3$ is a sheaf, and $(M_3, \CO_3)$ is a formal manifold.
	\end{lemd}
	\begin{proof}
		Let $U_1$ be an open subset of $M_1$.
		For every open subset $U_2$ of $M_2$, we have a natural homomorphism
		\be\label{o3}
		\CO_1(U_1)\widetilde \otimes \CO_2(U_2)=\CO_3(U_1\times U_2)\rightarrow \widetilde{\CO}_3(U_1\times U_2).
		\ee
		This yields a sheaf homomorphism
		\be\label{o32}
		\CO_{2, \CO_1(U_1)} \rightarrow (p_2)_*(\widetilde{\CO}_3|_{U_1\times M_2})
		\ee over $M_2$,
		where $ \CO_{2, \CO_1(U_1)} $ is the sheaf
		\[U_2\mapsto \CO_1(U_1)\widetilde \otimes \CO_2(U_2)\] on $M_2$  (see Theorem \ref{thmvvf}), and
		$p_2: U_1\times M_2\rightarrow M_2$ is the projection map. By Lemma \ref{prodvt}, the map
		\eqref{o3} is an isomorphism when $U_2$ is a chart. This implies that  \eqref{o32} is a sheaf isomorphism. In other words,
		\[
		\widetilde{\CO}_3(U_1\times U_2)=\CO_3(U_1\times U_2)
		\]
		for all open subsets $U_2$ of $M_2$. Then the sheaf property implies that $\widetilde{\CO}_3=\CO_3$, and Lemma \ref{prodvt} implies that $(M_3, \CO_3)$ is a formal manifold.

	\end{proof}

	\begin{lemd}\label{prodvt3}
		The presheaf $\CO_3$ is a sheaf, and $(M_3, \CO_3)$ is a formal manifold. %$\BC$-locally ringed space
	\end{lemd}
	\begin{proof}
		The proof is similar to that of Lemma \ref{prodvt2}. Let $U_2$ be an open subset of $M_2$.
		For every open subset $U_1$ of $M_1$, we have a natural homomorphism
		\be\label{o321}
		\CO_1(U_1)\widetilde \otimes \CO_2(U_2)=\CO_3(U_1\times U_2)\rightarrow \widetilde{\CO}_3(U_1\times U_2).
		\ee
		This yields a sheaf homomorphism
		\be\label{o322}
		\CO_{1, \CO_2(U_2)} \rightarrow (p_1)_*(\widetilde{\CO}_3|_{M_1\times U_2})
		\ee over $M_1$, 
		where $ \CO_{1, \CO_2(U_2)} $ is the sheaf \[U_1\mapsto \CO_1(U_1)\widetilde \otimes \CO_2(U_2)\] on $M_1$ (see Theorem \ref{thmvvf}), and
		$p_1: M_1\times U_2\rightarrow M_1$ is the projection map. By Lemma \ref{prodvt2}, the map
		\eqref{o321} is an isomorphism when $U_1$ is a chart. This implies that  \eqref{o322} is a sheaf isomorphism. In other words,
		\[
		\widetilde{\CO}_3(U_1\times U_2)=\CO_3(U_1\times U_2)
		\]
		for all open subsets $U_1$ of $M_1$. Thus $\widetilde{\CO}_3=\CO_3$, and   $(M_3, \CO_3)$ is a formal manifold by Lemma \ref{prodvt}.

	\end{proof}

	\begin{lemd}\label{prodvt4}
		The formal manifold $(M_3, \CO_3)$ is a product of $(M_1, \CO_1)$ and $(M_2, \CO_2)$ in the category of formal manifolds.
	\end{lemd}
	\begin{proof}
		By 
		%In view of
		Theorem \ref{thmqe},
		for every formal manifold $(M, \CO)$ we have that
		\begin{eqnarray*}
			&& \mathrm{Mor}((M, \CO), (M_3, \CO_3))\\
			&=& \Hom(\CO_3(M_3), \CO(M))\\
			&=& \Hom(\CO_1(M_1)\widehat \otimes \CO_2(M_2), \CO(M))\\
			&=& \Hom(\CO_1(M_1), \CO(M))\times \Hom(\CO_2(M_2), \CO(M)) \\
			&=& \mathrm{Mor}((M, \CO), (M_1, \CO_1))\times \mathrm{Mor}((M, \CO), (M_2, \CO_2)).
		\end{eqnarray*}
		Here $\mathrm{Mor}$ indicates the set of morphisms between two formal manifolds, and $\Hom$ indicates the set of continuous $\BC$-algebra homomorphisms.
		This proves the lemma.
	\end{proof}
	
	\begin{lemd}\label{prodvt5}
		In the category of smooth manifolds, we have that 
		\[
		(M_3,\underline{\CO_{3}})=(M_1, \underline{\CO_{1}}) \times (M_2, \underline{\CO_{2}}).
		\]
	\end{lemd}
	\begin{proof}
		Let $U_1\subset M_1$ and $U_2\subset M_2$ be open subsets, and set $U_3:=U_1\times U_2$.
		Recall from Example \ref{trivialf2} that  the reduction assignment
		\[
		(M, \CO)\mapsto (M, \underline \CO)
		\]
		is  a functor from the category of formal manifolds to the category of smooth manifolds.
		Thus the projection morphisms
		$(U_3,\CO_{3}|_{U_3})\rightarrow (U_1, \CO_{1}|_{U_1})$ and
		$(U_3,\CO_{3}|_{U_3})\rightarrow (U_2, \CO_{2}|_{U_2})$ induce respectively the smooth maps
		\[
		(U_3,\underline{\CO_{3}|_{U_3}})=(U_3,\underline{\CO_3}|_{U_3})\rightarrow (U_1,\underline{\CO_{1}|_{U_1}})=(U_1,\underline{\CO_1}|_{U_1})\]
		and\[
		(U_3,\underline{\CO_{3}|_{U_3}})=(U_3,\underline{\CO_3}|_{U_3})\rightarrow (U_2,\underline{\CO_{2}|_{U_2}})=(U_2,\underline{\CO_2}|_{U_2}).
		\]
		These two smooth maps produce a smooth map
		\be\label{defeo}
		(U_3,\underline{\CO_{3}}|_{U_3})\rightarrow (U_1, \underline{\CO_{1}}|_{U_1}) \times (U_2, \underline{\CO_{2}}|_{U_2}).
		\ee
		When $U_1$ and $U_2$ are charts, it follows from Lemma \ref{prodvt} that the smooth map \eqref{defeo} is a
		diffeomorphism.
		This implies  that the   map \eqref{defeo} is a
		diffeomorphism when $U_1=M_1$ and $U_2=M_2$.
	\end{proof}
	
	\begin{lemd}\label{prodvt6}
		For every $a_1\in M_1$ and $a_2\in M_2$, we have topological algebra  identifications
		\[
		\wh{\CO}_{3,a_3}=\wh{\CO}_{1,a_1}\wh{\otimes}\,\wh{\CO}_{2,a_2}=\wh{\CO}_{1,a_1}\wt{\otimes}\,\wh{\CO}_{2,a_2}\qquad (a_3:=(a_1,a_2)).
		\]
	\end{lemd}
	\begin{proof} The projection morphisms $(M_3,\CO_{3})\rightarrow (M_1, \CO_{1})$ and $(M_3,\CO_{3})\rightarrow (M_2, \CO_{2})$ induce
		respectively the continuous homomorphisms (see Lemma \ref{lem:homonformalstalkscon})
		\[
		\wh{\CO}_{1,a_1}\rightarrow \wh{\CO}_{3,a_3} \quad\textrm{and}\quad \wh{\CO}_{2,a_2}\rightarrow \wh{\CO}_{3,a_3}.
		\]
		Since $\wh{\CO}_{3,a_3}$ is complete, the above two maps yield a continuous homomorphism 
		\be\label{eq:formalstalktensor}
		\wh{\CO}_{1,a_1}\wh\otimes\, \wh{\CO}_{2,a_2}\rightarrow \wh{\CO}_{3,a_3}.
		\ee
		
		Note that we may (and do) assume that  $(M_1, \CO_{1})=(\R^{n_1},\CO_{\R^{n_1}}^{(k_1)})$ and  $(M_2, \CO_{2})=(\R^{n_2},\CO_{\R^{n_2}}^{(k_2)})$
		for some $n_1,n_2,k_1,k_2\in \BN$.
		Then by Lemma \ref{prodvt} we see that
		\[(M_3, \CO_{3})=(\R^{n_3},\CO_{\R^{n_3}}^{(k_3)})\quad (n_3:=n_1+n_2,\ k_3:=k_1+k_2).\]
		For  $i=1,2,3$, Taylor expansion at $a_i$ yields a topological algebra isomorphism (see  Proposition \ref{lem:formalstalkiso} and its proof)
		\[\wh{\CO}_{i,a_i}\cong \C[[x_1,x_2,\dots,x_{n_i},y_1,y_2,\dots,y_{k_i}]].\]
		This, together with Example \ref{ex:E-valuedpowerseries},  implies that $\wh{\CO}_{1,a_1}\wh{\otimes}\,\wh{\CO}_{2,a_2}$ $=\wh{\CO}_{1,a_1}\wt{\otimes}\,\wh{\CO}_{2,a_2}$, and that
		the map \eqref{eq:formalstalktensor} is a topological algebra isomorphism.

	\end{proof}
	
	Theorem \ref{pfm00} is now proved by combining Lemmas   \ref{prodvt3}, \ref{prodvt4}, \ref{prodvt5} and \ref{prodvt6}.

	\section*{Acknowledgement}
	F. Chen is supported by the National Natural Science Foundation of China (Nos. 12131018, 12161141001) and the Fundamental Research Funds for the Central Universities (No. 20720230020).
	B. Sun is supported by  National Key R \& D Program of China (Nos. 2022YFA1005300 and 2020YFA0712600) and New Cornerstone Investigator Program. 
	The first and the third authors would like to thank Institute for Advanced Study in Mathematics, Zhejiang University. Part of this work was carried out while they were visiting the institute.

\end{document}